\documentclass[a4paper,UKenglish]{lipicsmio}

\usepackage{color}

\usepackage{microtype}
\usepackage{bussproofs}
\usepackage{upgreek}
\usepackage{hyperref}
\usepackage{tikz}
\usepackage{stmaryrd}
\usepackage{MnSymbol}

\title{Game Semantics and the Geometry of Backtracking: a New Complexity Analysis of Interaction
}
\titlerunning{Game Semantics and the Geometry of Backtracking} 

\author{Federico Aschieri \footnote{This work was funded by the Austrian Science Fund FWF Lise Meitner grant M 1930--N35 and previously supported by the FWF grant P25160--N25.}}

\affil{Institut f\"ur Diskrete Mathematik und Geometrie\\Technische Universit\"at Wien\\ Wiedner Hauptstra\ss e 8-10/104, 1040, Vienna, Austria}
\authorrunning{} 

\Copyright[by]{}

\newcommand{\comment}[1]{}

\newcommand{\EM}                       { {\mathsf{EM}} }

\newcommand{\pview}[1]   {{\ulcorner #1 \urcorner}}
\newcommand{\oview}[1]   {{\llcorner #1 \lrcorner}}
\newcommand{\efrom}[1]   {{\nwarrow #1}}
\newcommand{\cross}    {{\, \curvearrowleft\, }}
\newcommand{\blev}          {{\mathsf{Blevel}}}
\newcommand{\inact}          {{\,\medtriangleleft\,}}

\newcommand{\com}[2]    {{\mathsf{comp}(#1, #2)}}
\newcommand{\comu}[2]    {{\mathsf{comp}_{#1}(#2)}}
\newcommand{\trim} [2]         {{#2}^{#1}}
\newcommand{\on}         {{\varowedge}}
\newcommand{\off}         {{\varovee}}

\theoremstyle{plain}
\newtheorem{proposition}[theorem]{Proposition}

\begin{document}
\maketitle
\begin{abstract}
We present abstract complexity results about Coquand and Hyland-Ong game semantics, that will lead to new bounds on the length of first-order cut-elimination, normalization, interaction between expansion trees and any other dialogical process game semantics can model and apply to. In particular, we provide a novel method to bound the length of interactions between visible strategies and to measure precisely the tower of exponentials defining the worst-case complexity.  Our study improves the old estimates on average by several exponentials. 

\end{abstract}
\section{Introduction}

Of all towers of the world, the one we would contemplate far more admiringly if it were lower rather than higher, is a tower in logic. It is the stack of exponentials defining the worst-case complexity of cut-elimination, normalization and dialogues between strategies interpreting  first-order proofs \cite{Beckmann, Schwichtenberg}:
$$\underbrace{2^{2^{.^{.^{2^{k}}}}}}_{r+1}$$
The tower's height is measured by the rank $r$ of the proof, which is in sequent calculus the highest among the logical complexities of the proof cut-formulas and in natural deduction the highest among the type complexities of the proof redexes; $k$ is just  the proof height. Given the use of complex cuts and redexes in logical derivations, it does not take too much to reach astronomical worst-case bounds.

Yet, thanks to the famous Curry-Howard correspondence \cite{Wadler}, we know that logical proofs, both in natural deduction \cite{AschieriZH,Krivine,Parigot} and in sequent calculus \cite{Herbelin,CurienHerbelin}, can be seen as programs of some typed functional programming language. And in the New York of a programming language, many inhabitants are indeed to be found on the highest floors of the highest skyscrapers, but the vast majority of them live in far lower buildings. That's programming: one can write enormously complex programs, but also very simple and efficient ones. 
Having complexity bounds only for the very worst case of the very worst program is of no use when one looks for a more precise idea of how much time it takes a particular program to terminate its execution. Unless one is terribly unlucky, his actual programs will take much less time than the worst possible program to complete their computations. How much less? We need to measure precisely the tower of exponentials, because it is its height that is to blame for the unrealistic bounds it computes.

The reason why the worst-case analysis of cut-elimination fares so roughly on average is the  complexity measure adopted: the rank of the proof is a superficial information, in a literal sense.
Indeed, one looks only at the surface of a cut, the cut-formula, without caring about how the formula is used in the proof. It is like trying to determine the maximum speed of a car just by studying its outward appearance, while ignoring completely its internal engine! But what is the engine of a cut? What is the real responsible of the worst-case complexity bounds, the mechanism that determines the height of the tower of exponentials?
\vspace{-3ex}
\subsection{Coquand and Hyland-Ong Game Semantics}
It turns out that the laws governing the complexity of cut-elimination are written in a game semantical language. If we want to  answer  our questions, we have no choice: we must translate them in terms of games and strategies. In the quest for a better complexity measure of cut-elimination and normalization, the single most important tool is then \emph{game semantics}. Game semantics was in the background of logic at least since Herbrand's Theorem: this is clear by Miller's formulation of it using expansion trees \cite{Miller}, that is, game strategies, as remarked by Heijltjes \cite{Heijltjes}. But the idea of interpreting formulas as games and proofs as winning strategies, appeared time and again in the history of logic: in Gentzen's and Novikoff's proofs of consistency of Arithmetic (see \cite{Coquand}), in G\"odel's and Kreisel's functional interpretations (see \cite{Kreiselm, AschieriZMP, OlivaPowell}), in Lorenzen's game semantics (see \cite{Felscher}).

 The old, \emph{statical} game semantics was brought to a new life with the invention of the modern, \emph{dynamical} game semantics, initiated by Coquand's \cite{Coquand0, Coquand} discovery that cut-elimination in Arithmetic mirrors
  a debate between two strategies. As opposed to ordinary games, players may every now and then revise their moves and backtrack to a previous position. Coquand proved by elegant combinatorial means that the dialogues between strategies interpreting proofs always terminate, thus proving as well the termination of cut-elimination, as  later shown by Herbelin \cite{Herbelin}. 
 The advantage of game semantics is that it ignores inessential aspects of proofs and turns the spotlight to their geometrical structure.


Hyland and Ong \cite{HO} took Coquand-style game semantics to a new level. The key move is realizing that one can drop altogether winning conditions:
debates need not always terminate with the win or the loss of one of the two players. Game semantics can in such a way model pure dialogues, exchanges of answers and questions, as those occurring during the evaluation of a functional program. The linear head reduction strategy of typed lambda terms can thus be modeled as a debate between strategies \cite{DRHerbelin}.
Thanks to these and several other insights, Hyland and Ong discovered how to construct fully abstract models for typed functional languages  (see also \cite{Abramsky, AbramskyMcCusker}). 

As plays correspond to cut-elimination \cite{Herbelin} and normalization \cite{DRHerbelin}, one encounters in game semantics the same complexity estimates.
The rank of a cut-formula becomes the depth of the game arena, and the very same tower of exponentials that bounds cut-elimination, bounds also the dialogues between strategies corresponding to first-order proofs. 
The result was proved by Clairambault \cite{Clairambaultconf,Clairambaultjour};  it constitutes an important demonstration that game semantics may serve as a viable complexity analysis tool.  Nevertheless, the depth of game arenas is not a precise complexity measure for the same reasons we discussed above: it corresponds just to the rank of cut-formulas. 

It  may appear that game semantics is leading us nowhere. For the moment we hardly achieved anything but to translate doubts and questions from one language to another. What is the engine of a strategy? How to measure its complexity? 
 We are missing a piece of the puzzle and to get it, we need to make two long jumps: one backward and the other forward.
\vspace{-2ex}
\subsection{Back to Coquand's Simple Backtracking}
Debates in game semantics can be extremely involved; proving their termination, like escaping from an intricate labyrinth with no Ariadne's thread at hand.
Coquand eventually overcame these difficulties by inventing a notation with deep geometrical properties: pointer sequences.
 But in a preliminary effort  \cite{Coquand0} to establish his successive termination result \cite{Coquand}, he settled for a weaker one: that debates between \emph{simple} backtracking strategies, now called \emph{$1$-backtracking strategies}, always terminate. 
According to Coquand, a $1$-backtracking strategy ``never changes its mind about a value it has previously considered wrong''. This kind of strategies are by design less involved than general ones and their debates are more predictable; that allowed Coquand to get away at least with some result from his first struggle. 
\vspace{-2ex}
\subsection{Berardi-de'Liguoro Backtracking Level of a Strategy}

The idea of classifying strategies according to some complexity measure had an extraordinary potential, yet it was completely abandoned.  Only much later, did Berardi and de'Liguoro \cite{BerBack}  generalize the concept of $1$-backtracking to the more general \emph{$n$-backtracking} for any $n\in\mathbb{N}$. It was a major conceptual advancement.
Their definition is inspired by Gold's learning theory \cite{Gold,Jain} and by the  Berardi-Coquand-Hayashi \cite{BerCoqHay} notion of $n$-backtracking game  -- an alternative to Coquand and Hyland-Ong games up to now only partially explored. However, it was not at all obvious how to formulate a notion of $n$-backtracking strategy in Coquand and Hyland-Ong games as well, which is what they did.\\
The definition is logically complex and has a geometrical and temporal flavour, but its meaning is simple to grasp. For example, a strategy is $2$-backtracking if it is allowed to change its mind about previous mind changes of level $1$. That reminds us of some juridical systems: only a judge of second level can change the decision of a judge of level one. Thus a judge of an appeal court can very well revise an ordinary judge's decision, but two ordinary judges cannot question each other's sentences. With Coquand-Hyland-Ong pointer notation, a mind change of level $1$ is represented by a move pointing backward: 
$$\begin{tikzpicture}
\node (a) at (0,0) {$\bullet_1$};
\node at (1,0)  {$\ldots$};
\node (c) at (2,0) {$\circ_1$};
\draw[-latex,bend right]  (c) edge (a);
\end{tikzpicture}$$
\vspace{-1.5ex}
A mind change of level $2$ can be depicted as a crossing: 
$$
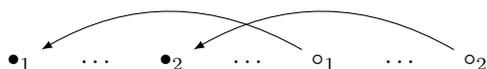
\begin{tikzpicture}
\node (a) at (0,0) {$\bullet_1$};
\node at (1,0)  {$\ldots$};
\node (b) at (2,0) {$\bullet_2$};
\node at (3,0)  {$\ldots$};
\node (c) at (4,0) {$\circ_1$};
\node at (5,0)  {$\ldots$};
\node (d) at (6,0) {$\circ_2$};
\draw[-latex, bend right]  (d) edge (b);
\draw[-latex,bend right]  (c) edge (a);
\end{tikzpicture}$$
Though they are \emph{not} physically removed, we consider the moves between the ending points of a pointer to be \emph{ideally} erased by the mind change. Thus, the move $\circ_1$ is a first mind change erasing the moves before itself and after $\bullet_{1}$; the move $\circ_2$ \emph{is a mind change about the first mind change} and recovers one of the formerly retracted moves, $\bullet_{2}$. There is a proviso: for the move $\circ_{2}$  to be considered a \emph{real} mind change of level $2$, at the moment $\circ_{2}$ is played the pointer from $\circ_{1}$ must be \emph{active}: a pointer is active if it is crossed only by inactive pointers; a pointer is inactive if there is an active pointer crossing it. By considering mind changes about mind changes about mind changes... one arrives at the concept of $n$-backtracking. A $n$-backtracking strategy has backtracking level $b=n$. 
We shall have to describe precisely the notion and familiarize extensively with it in the following. For the moment it is enough to say that the backtracking level is a measure of how much involved and intricate the geometrical structure of  a strategy is. The backtracking level is also easy to compute: if a strategy interprets a first-order proof or a simply typed lambda term, one can determine its backtracking level in plain quadratic time, and we conjecture linear. 
\vspace{-2ex}
\subsection{Backtracking Level as Complexity Measure}

We have  found the complexity measure we were looking for: it is the backtracking level of a strategy. We shall prove that the height of the tower of exponentials bounding the length of interactions between strategies is determined by the backtracking level of the strategies. The explanation of this complexity result is our paper's\\

\noindent\textbf{Fundamental Thesis}: \emph{the real engine of a strategy is its backtracking mechanism; the power of this engine is measured by the strategy's backtracking level; the higher the backtracking level, the more involved the geometrical structure of the backtracking mechanism is and the longer it will take to conclude a debate with the strategy}.\\

\noindent As a consequence of our thesis, at first we conjectured a weaker result than the one presented here, which was: if we consider two strategies with backtracking level respectively $n$ and $m$, $b=\mathsf{max}(m,n)$, and the size of the two strategies is bounded by $k$, then the complexity bound for the length of any interaction between them is 
\vspace{-1.5ex}
$$\underbrace{2^{2^{.^{.^{2^{k(\log k)2}}}}}}_{b+1}$$
This estimate already outperforms
the old ones, because the backtracking level of a strategy is completely independent of the arena depth. A $2$-backtracking strategy can live in an arena of depth $1000$; then one has a  new bound given by a tower of exponentials of height $3$ versus an old one of height $1001$. Moreover, as we shall see, the backtracking level of a strategy is bound by the depth of the arena, modulo a trivial transformation. That explains perfectly why the rank of a proof worked as a worst-case complexity measure; in a sense, one looked at the right place, without knowing the good reason. 

\noindent Our fundamental thesis seemed  to suggest that our first conjecture was already optimal. For example, one would expect that when a $5$-backtracking strategy interacts with a $1$-backtracking one, the debate should take in general no less time than prescribed by a tower of exponential of height $6$. \\ 
We were struck when we found that the case it's not. 
The complexity bound in this situation is much tighter: a tower of height $2$, that is just an exponential
\vspace{-1.5ex}
$$2^{(k\log k)2}$$
And $1$-backtracking strategies are no  fortunate exceptions. If a $5$-backtracking strategy  interacts with a $2$-backtracking one, the complexity bound  is just a tower of height $3$, a double exponential:
\vspace{-1ex}
$$2^{2^{(k \log k) 2}}$$
More in general, given two strategies of backtracking level respectively $n$ and $m$ and of size less than $k$, we shall prove that the length of their interaction is determined by the \emph{minimum} among their backtracking levels, that is by $b=\mathsf{min}(m,n)$, and will be less than
\vspace{-1.5ex}
$$\underbrace{2^{2^{.^{.^{2^{k(\log k)2}}}}}}_{b+1}$$
We can interpret this result as a \emph{safety of interactions} result. No matter how unpredictable and complex the context in which a strategy is thrown in, the strategy will always have safe interactions, from the complexity point of view, whose rate of growth is determined by the strategy's own backtracking level. 
\vspace{-2ex}
\subsection{$n$-Backtracking, $n$-Quantifiers Excluded-middle $\EM_{n}$ and Learning}

One may acknowledge that indeed the concept of $n$-backtracking provides the means for a very refined analysis of the length of dialogues in logic, but still have the suspicions that the analysis is of no practical consequence. Are there around enough $1$-backtracking strategies, $2$-backtracking strategies and so on, or are they odd entities appearing, like the 29th of February, once in every four years of logical and mathematical practice?
Are almost all formulas of complexity $n$  going to have only $n$-backtracking winning strategies? The answer is no. In Arithmetic the landscape is already partially studied, thanks to the correspondence between the excluded middle $\EM_{n}$ over formulas with $n$ quantifiers and the concept of learning. In \cite{Aschierigames} it was shown, by means of Interactive realizability \cite{ABlmcs}, that every arithmetical proof using only $\EM_{1}$ as classical principle represents a winning recursive strategy for the related $1$-backtracking game in the sense of \cite{BerCoqHay}. In \cite{BerCoqHay} was proved -- in a less refined way -- that any  arithmetical formula derivable using $\EM_{n}$ as a classical principle, has a winning recursive strategy in $n$-backtracking games. As these results are very likely to hold also for first-order logic in the formulation of \cite{AschieriZH}, it is clear that strategies of low backtracking level are the norm rather than the exception. Moreover, as a matter of fact many theorems in Arithmetic are proved just by using $\EM_{1}$ and many theorems in Analysis with $\EM_{2}$.

The results of this work also definitely confirm a point that has been made several times \cite{ABlmcs,AschieriHAS,BerardiKnow}: proofs limited to $\EM_{1}$ or $\EM_{2}$ yield relatively  efficient strategies, and in general the \emph{learning theory}  is the correct one to understand and improve the complexity of programs  extracted from classical proofs. The surprise, here, is the connection with the complexity of intuitionistic proofs as well.   
\vspace{-2ex}
\subsection{Content of the Paper}
In this paper we establish the discussed bounds on the interaction lengths between visible bounded strategies in Coquand and Hyland-Ong game semantics. 
They will certainly lead to new bounds also for cut-elimination, normalization and dialogues between expansion trees in first-order logic; but given the great generality of game semantics, our results are interesting in their own right. For space reasons, we will not study in detail concrete applications and we delay the matter to forthcoming papers. As explained in \cite{Clairambaultconf,Clairambaultjour, DRHerbelin}, however, bounds in game semantics automatically translate into bounds for linear head reduction in typed lambda calculus. There is also Herbelin's result \cite{Herbelin}, showing a step-by-step correspondence between plays in game semantics and the weak head reduction cut-elimination strategy in an infinitary sequent calculus for Arithmetic. The situation in first-order logic is a bit subtler, because atomic formulas are undecidable, thus winning conditions must be different; yet we do not see major obstacles to adapting Herbelin's methods also to that setting. 

We devote the introductory part of the paper to recalling the game semantical ideas of Coquand, Hyland-Ong and Berardi-de'Liguoro, with new examples and viewpoints that we feel are essential to understand the technical part of the paper. In particular, in Section \S\ref{sec-CHO}, we describe the main concepts of Coquand and Hyland-Ong game semantics and study an example based on Miller's expansion trees. In Section \S\ref{sec-BacktrackingLev}, we explain in detail the Berardi-de'Liguoro notion of backtracking level of a strategy. In Section \S\ref{sec-proofidea}, we outline our proof plan, which is carried out in detail in Sections \S\ref{sec-maxmin}, \S\ref{sec-trimm}, \S\ref{sec-bounds}, \S\ref{sec-gadget}.

\vspace{-2ex}
\section{Coquand-Hyland-Ong Game Semantics}\label{sec-CHO}
As Wadler \cite{Wadler} put it, ``Curry-Howard is a double-barrelled name which ensures the existence of other double-barrelled names''. Since logical proofs are isomorphic to computer programs, many concepts happen to be discovered twice: once in logic, once in computer science. Girard-Reynolds system $F$ is an example, Coquand-Hyland-Ong game semantics is another. Actually, most literature on game semantics fails to mention that Coquand and Hyland-Ong game semantics are by no means different concepts but just \emph{one} semantics that turns out to have several applications:
cut-elimination in the case of Coquand, denotational semantics for typed lambda calculus  in the  case of Hyland-Ong. Since the latter has in turn as many variants as functional languages out there, to avoid the \emph{one, none, one hundred thousand} Pirandellian fragmentation, we adopt the most abstract Hyland-Ong terminology (see \cite{Clairambaultjour}), calling it Coquand-Hyland-Ong game semantics.


Arenas are the structures where games take place. As in many real-life two-player games, there are a set of moves $M$ and a binary relation of justification $\vdash$ between moves. The elements of $M$ may represent elementary actions of a player, like moving a single piece in chess or choosing a witness in logical dialogues; but they may also encode complex positions, like the entire chessboard or a full logical formula. Often a game can be formalized in both ways. The relation $\vdash$ tells whether a move is a direct answer to another one or whether from a configuration of the game a player can pass to another. Every move is labeled by a function $\lambda$ as either Player or Opponent move, and there are special initial moves that start the game.    


\begin{definition}[Arenas]
An \textbf{arena}  $\mathbb{A}$ is a structure $(M, \vdash, \lambda, I)$ such that:
\begin{itemize}
\item
$M,\vdash$ is a graph: $M$ is a set of nodes, called \textbf{moves}, and $\vdash$ is a binary relation over $M$, called \textbf{enabling}. If $m\vdash n$, then $m$ is said to \textbf{justify} $n$ and $n$ is said to \textbf{answer} $m$. 
\item $\lambda: M\rightarrow \{P,O\}$ is a function indicating whether any move $m$ is a Player move or an Opponent move and is such that for every $m,n\in M$, $m\vdash n \implies \lambda(m)\neq \lambda(n)$, which is called the \textbf{alternation condition}.
\item $I$ is a subset of $M$ called the set of \textbf{initial moves} and is such that for every $m\in I$, $\lambda(m)=O$. We assume $*\in I$.
\item A \textbf{simple play} over $\mathbb{A}$ is any sequence of moves $m_{1}\, m_{2}\, m_{3} \ldots m_{i}$ such that $m_{1}\in I$ and $$m_{1}\vdash m_{2} \vdash m_{3}\vdash\cdots \vdash m_{i}$$ and represents a backtracking free play.
\item $\mathbb{A}$ is said to be of \textbf{finite depth} whenever there exists a $k\in\mathbb{N}$ such that no simple play over $\mathbb{A}$ is of length greater than $k$; in this case, $k$ is said to be the \textbf{depth} of $\mathbb{A}$. \\
\end{itemize}
\end{definition}
When we shall want to talk about a player without specifying who he is, we will call him  player $p,q, \ldots$ and the other player will be $\overline{p}, \overline{q}, \ldots$; that is, if $p=P$, then $\overline{p}=O$ and if $p=O$, then $\overline{p}=P$. Often moves will be denoted with  white $\circ$ or black $\bullet$ circles, that may be indexed; in a play, Player and Opponent moves will always be represented with different colors.  Concatenations of sequences $s$ and $s'$ of moves will just be denoted as $s\, s'$. An \textbf{initial segment} of a sequence $s_{0}\, s_{1}\ldots s_{k}$ is any prefix  $s_{0}\, s_{1}\ldots s_{i}$, with $0\leq i \leq k$.

So far, we have described just a standard game: we have considered only \emph{simple} plays, where players move in an alternating fashion and each player immediately answers the last move of the opponent. A simple play in chess would be: $*$1.e4 e5 2.Nf3 Nf6 3.Nxe5, which means moving in order a white pawn to e4, a black pawn to e5, a white knight on f3, a black knight on f6 and taking the pawn on e5 with the white knight.  Graphically, with arrows representing the justification relation $\vdash$, one has:
\vspace{-1ex}
$$\begin{tikzpicture}
\node (0) at (-1, 0) {$*$};
 \node (x) at (0,0) {e4};
\node (a) at (1.5,0) {e5};
\node (b) at (3,0) {Nf3};
\node (c) at (4.5,0) {Nf6};
\node (d) at (6,0) {Nxe5};
\draw[-latex, right]  (x) edge (0);
\draw[-latex, right]  (a) edge (x);
\draw[-latex, right]  (b) edge (a);
\draw[-latex, right]  (c) edge (b);
\draw[-latex, right]  (d) edge (c);
\end{tikzpicture}$$
When interpreting dialogues in logic and computer science, however, it is necessary to allow players to delay answers, take back moves, backtrack to a previous opponent move and answer differently. All of that  can be represented by an act of answering a previous move with a pointer. For example, black, instead of answering the move 3.Nxe5, may change his mind and answer differently the move 2.Nf3 by placing his knight to c6:
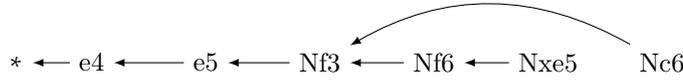
\vspace{-2ex}
$$\begin{tikzpicture}
\node (0) at (-1, 0) {$*$};
 \node (x) at (0,0) {e4};
\node (a) at (1.5,0) {e5};
\node (b) at (3,0) {Nf3};
\node (c) at (4.5,0) {Nf6};
\node (d) at (6,0) {Nxe5};
\node (e) at (7.5, 0) {Nc6};
\draw[-latex, right]  (x) edge (0);
\draw[-latex, right]  (a) edge (x);
\draw[-latex, right]  (b) edge (a);
\draw[-latex, right]  (c) edge (b);
\draw[-latex, right]  (d) edge (c);
\draw[-latex, bend right]  (e) edge (b);
\end{tikzpicture}$$
Globally, there are now two plays on the board: $*$1.e4 e5 2.Nf3 Nf6 3.Nxe5 and $*$1.e4 e5 2.Nf3 Nc6. The play might go on like this
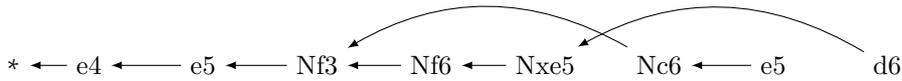
\vspace{-2ex}
$$\begin{tikzpicture}
\node (0) at (-1, 0) {$*$};
 \node (x) at (0,0) {e4};
\node (a) at (1.5,0) {e5};
\node (b) at (3,0) {Nf3};
\node (c) at (4.5,0) {Nf6};
\node (d) at (6,0) {Nxe5};
\node (e) at (7.5, 0) {Nc6};
\node (f) at (9, 0) {e5};
\node (g) at (10.5, 0) {d6};
\draw[-latex, right]  (x) edge (0);
\draw[-latex, right]  (a) edge (x);
\draw[-latex, right]  (b) edge (a);
\draw[-latex, right]  (c) edge (b);
\draw[-latex, right]  (d) edge (c);
\draw[-latex, bend right]  (e) edge (b);
\draw[-latex, right]  (f) edge (e);
\draw[-latex, bend right]  (g) edge (d);
\end{tikzpicture}$$
The plays  on the board are now: $*$1.e4 e5 2.Nf3 Nf6 3.Nxe5 d6 and $*$1.e4 e5 2.Nf3 Nc6 3.e5.
We can see in the example above the two different characters  that backtracking moves may possess:  the move d6 is a delayed first answer to Nxe5 and represents a mind change about which game line to play; the move Nc6 is a second answer to Nf3 and represents a mind change about how to answer a move that was already answered. As for our purposes it would be useless to take account of that difference, we shall regard each backtracking move as a mind change, without further distinctions.\\
Now we turn to a crucial concept, that of \emph{view}: a backtracking strategy needs to know what part of the game history to consider when deciding a new move. In any game, what really matters from the perspective of a player is: starting with a move, waiting for an answer,  making a move, waiting for an answer, making a move, waiting for an answer and so on...
The kind of strategy that arises from the interpretation of proofs and programs, just needs this sequence of move-answer, move-answer, move-answer... in order to decide the next move. The view of white in the last game above is thus: $*$ e4 e5 Nf3 Nf6 Nxe5 d6. In general, to determine the current view, a player looks at the last move of the opponent, then follows the pointer to the move that was answered, then consider the immediately preceding move, then follows the pointer to another move and so on... until an initial move is reached. In the picture
$$\begin{tikzpicture}

\node at (3,0.03)   {$\bullet$};
\node  (z) at (3.5,0) {$ \circ_{z}$};
\node  at (4,0) {$\ldots$};
\node (bk) at (4.5,0) {$\bullet_{z}$};
\node  at (5,0) {$\ldots$};
\node (ck-1) at (5.5,0) {$\circ_{3}$};
\node (z2) at (6.15,0) {$\ldots$};
\node (bk-1) at (6.7,0) {$\bullet_{3}$};
\node (c2) at (7.2,0) {$\circ_{2}$};
\node (z2) at (7.7,0) {$\ldots$};
\node (b2) at (8.2,0) {$\bullet_{2}$};
\node  (c1) at (8.7,0) {$\circ_{1}$};
\node  at (9.2,0) {$\ldots$};
\node  (b1) at (9.7,0) {$\bullet_{1}$};
\draw[-latex,bend right]  (b1) edge (c1);
\draw[-latex,bend right]  (b2) edge (c2);
\draw[-latex,bend right]  (bk) edge (z);
\draw[-latex,bend right]  (bk-1) edge (ck-1);
\end{tikzpicture}$$
the displayed moves, where $\bullet$ may or may not appear, represent the view of the player with the obligation to play after the last move $\bullet_{1}$.

\begin{definition}[Plays, Views]
Let $\mathbb{A}=(M, \vdash, \lambda, I)$ be an arena. 
\begin{itemize}
\item A finite sequence $s=s_{0}\, s_{1}\,\ldots\, s_{n}$ of moves of $M$ is said to be a \textbf{justified sequence} over $\mathbb{A}$ if $s_{0}$ is initial and every move $s_{i}$ of $s$, with $i>0$, is equipped with a pointer to a previous move $s_{j}$ justifying it: that is, $j< i$ and $s_j\vdash s_i$. 
\item A justified sequence $s=s_{0}\, s_{1}\,\ldots\, s_{n}$ is said to be a \textbf{play} over $\mathbb{A}$ if satisfies the alternation condition: for every $0\leq i< n$, $\lambda(s_i)\neq \lambda(s_{i+1})$. 

\item The pairs of moves $e=(\bullet, \circ)$ such that $\bullet\vdash \circ$ are called \textbf{edges}; $e$ is said to be the \textbf{edge from $\circ$} and is denoted as $\efrom{\circ}$; if $\lambda(\circ)=P$, $e$  is a \textbf{P-edge}, if $\lambda(\circ)=O$, $e$ is an \textbf{O-edge}.

\item The  \textbf{Player-view} $\pview{s}$ of a justified sequence $s$ is defined recursively as follow:
\begin{itemize}
\item $\pview{s\, \circ}=\pview{s}\, \circ\ $, if $\lambda(\circ)=P$.
\item $\pview{s\, \circ \, s'\, \bullet}=\pview{s}\, \circ\, \bullet\ $, if $\lambda(\bullet)=O$ and $\bullet$ points to $\circ$.
\item $\pview{\bullet}=\bullet$, if $\bullet\in I$ .
\end{itemize}
Therefore, whenever $s=\bullet_{0}\, \circ_{1}\, s_{1}\, \bullet_{1}\,  \circ_{2}\, s_{2}\, \bullet_{2}\ldots  \circ_{k}\, s_{k}\, \bullet_{k}$ and $\bullet_{0}$ is an initial move and for each $i\geq 1$, $\bullet_{i}$ is an Opponent move pointing to $\circ_{i}$, then $\pview{s}=\bullet_{0}\,\circ_{1}\, \bullet_{1}\,  \circ_{2}\, \bullet_{2}\ldots \circ_{k}\, \bullet_{k}$.

\item The \textbf{Opponent-view} $\oview{s}$ of a justified sequence $s$ is defined recursively as follow:
\begin{itemize}
\item $\oview{s\, \circ}=\oview{s}\, \circ\ $, if $\lambda(\circ)=O$.
\item $\oview{s\, \circ \, s'\, \bullet}=\oview{s}\, \circ\, \bullet\ $, if $\lambda(\bullet)=P$ and $\bullet$ points to $\circ$.
\item $\oview{s}=\epsilon$, if $s$ is the empty sequence $\epsilon$.
\end{itemize}
Therefore, whenever $s=\bullet_{0}\, s_{0}\,  \circ_{0}\, \bullet_{1}\, s_{1}  \circ_{1}\ldots  \bullet_{k}\, s_{k}\, \circ_{k}$ and $\bullet_{0}$ is an  initial move and for each $i\geq 0$, $\circ_{i}$ is a Player move pointing to $\bullet_{i}$, then $\oview{s}=\bullet_{0}\,\circ_{0}\, \bullet_{1}\,  \circ_{1}\ldots \bullet_{k}\, \circ_{k}$.\\

\end{itemize}
\end{definition}
Every play $s$ over  an arena $\mathbb{A}$ comes with a set of edges, representing the justification pointers. We say that an edge $e$ is in $s$ whenever $e$ is among the edges of $s$. As standard in the literature, we will not bother with how, formally, justification pointers are represented; it could be, for example, with a function $f: \mathbb{N}\rightarrow \mathbb{N}$ such that $f(x) < x$ for every $x>0$. However, we will just consider plays as sequences of moves with arrows pointing back to a previous move, as in the examples that we consider throughout the paper. In general, when taking subsequences of a play, we will maintain the arrows of the original play. We remark that, according to our definition, in any justified sequence every move different from the first points back to a previous move, like in \cite{Coquand, HO}; in particular, initial moves that occur in the middle of the play must obey the rule. There is no real technical need to enforce this restriction, but it is useful, for it allows to treat uniformly backtracking moves as those pointing before the immediately preceding move. Otherwise, an initial move in the middle of a play that does not point to anything should still be considered a backtracking to the start of the play, but this would not be recorded by an explicit pointer. In some literature, that behaviour is permitted;  but one can always put in front of  any play a pair of dummy moves, ``Let's begin'' and ``I'm ready'', the second of which justifies all Opponent moves, so every formerly initial move can point back to it.

 The definition of strategy is standard: it is a set of plays, those considered to ``follow'' the strategy; alternatively, a strategy for a player can be interpreted as a function, mapping any play $s$ ending with an adversary move to some suggested answers. A \emph{bounded} strategy is the kind of strategy that is supposed to interpret a simply typed lambda term or a first-order proof: there is a bound on the length that the view of any play in the strategy can have. The idea is that these strategies always stop playing after  a number of moves below a bound which is known in advance. In this paper we shall only be concerned with \emph{visible} strategies: those only needing a view as the history to refer back to with their answers. Visible sequences represent interactions between visible strategies. 
\begin{definition}[Visibility, Strategies, Interactions]\label{def-stratvisint}
Let $\mathbb{A}=(M, \vdash, \lambda, I)$ be an arena. 
\begin{itemize}
\item Let $s=s_{0}\, s_{1}\, \ldots s_{n}$ be a play over $\mathbb{A}$. $s$ is said to be \textbf{visible} if for all $0<i\leq n$, $s_{i}$ points to a move in $\pview{s_{0}\, s_{1}\, \ldots s_{i-1}}$ when $\lambda(s_{i})=P$ and $s_{i}$ points to a move in $\oview{s_{0}\, s_{1}\, \ldots s_{i-1}}$ when $\lambda(s_{i})=O$.
\item A \textbf{Player strategy} over $\mathbb{A}$ is any set  $\sigma$ of even-length plays over $\mathbb{A}$ such that if $s\, m\, n\in \sigma$, then $s\in \sigma$. $\sigma$ is said to be \textbf{visible} if for every $s\, m\in\sigma$, $m$ points to a move in $\pview{s}$.

\item An \textbf{Opponent strategy} over $\mathbb{A}$ is any set  $\sigma$ of odd-length plays over $\mathbb{A}$ such that if $s\, m\, n\in \sigma$, then $s\in \sigma$. $\sigma$ is said to be \textbf{visible} if for every $s\, m\in\sigma$, with $s$ non empty, $m$ points to a move in $\oview{s}$.

\item A visible Player strategy $\sigma$ is said to be \textbf{bounded} by $k$ if $k\in\mathbb{N}$ and for no $s\in\sigma$,  $\pview{s}$  is of length greater than $k$. A visible Opponent strategy $\sigma$ is said to be \textbf{bounded} by $k$ if $k\in\mathbb{N}$ and for no $s\in\sigma$,  $\oview{s}$  is of length greater than $k$. 

\item Let $\sigma$ and $\tau$ be respectively a strategy for Player and a strategy for Opponent over $\mathbb{A}$. The set of \textbf{interactions} $\sigma\star\tau$ is defined as 
$$\sigma \star \tau:=\{s\, m\ |\ (\lambda(m)=P\implies s\, m\in\sigma \land s\in\tau )\land (\lambda(m)=O\implies s\, m\in\tau\land s\in\sigma) \}$$ 
It represents a match between $\sigma$ and $\tau$, i.e. a play which follows both the strategies. \\
\end{itemize}
\end{definition}
The goal of this paper is to study the length of the interactions in $\sigma\star\tau$, whenever $\sigma$ and $\tau$ are bounded: by Coquand's termination Theorem \cite{Coquand}, in that case, $\sigma\star\tau$ contains only finite sequences. But first an example. 
\subsection{Example: Space, Time and Expansion Trees}

Game semantics does more than just describe debates between strategies;
it also sheds light on the roles that the very concepts of space and time play in proof representation. 
But what are space and time in logic?

Conventional wisdom suggests that proofs, on one hand, have a geometrical structure which is important to investigate: after all, a proof is a tree in a two dimensional \emph{space}. On the other hand, conventional wisdom suggests that the temporal process of constructing a proof is as important as its geometrical appearance: proofs pertain also to \emph{time}. Time is no issue: the act itself of drawing an inference takes place in time and to each logical step is possible to associate the moment in which it was made. But where exactly is the geometry of a proof to be sought? Here conventional wisdom is wrong. 

Think about a cut-free proof of a prenex formula in a one-sided sequent calculus. The geometrical structure of the proof tree is trivial: it is a linear sequence of $\exists$ and $\forall$ inferences introducing respectively first-order terms and eigenvariables; after those introductions, a propositional tautology is proved, but at that very point the proof becomes no longer interesting! The geometrical structure of a proof does not coincide with its tree structure: it is concealed in expansion trees.

Miller's expansion trees \cite{Miller} allow to formulate Herbrand's Theorem for what it is: a \emph{game semantical interpretation of proofs}. For each cut-free Tait-style sequent calculus proof of a sequent $\Gamma, A$, one can trace how the formula $A$ is used in the proof, keeping note  of what first-order terms $t$ the existential quantifiers of $A$ are instantiated to and what eigenvariables $a$ the universal quantifiers of $A$ introduce. The result is an expansion tree. For example, the formula $\exists x\,\forall y\,\exists z\, P(x,y,z)$ may have the following expansion tree:
$$
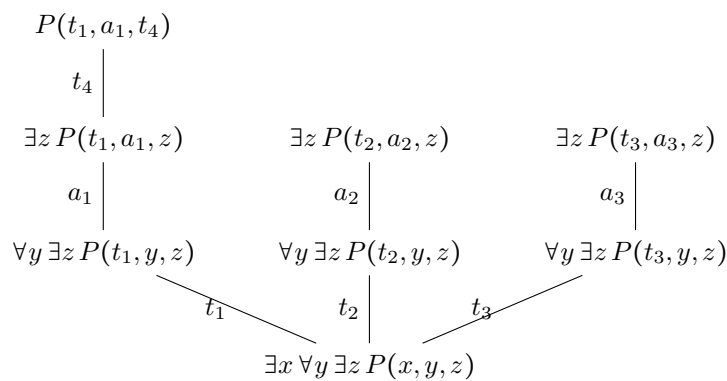
\begin{tikzpicture}
\tikzstyle{level 1}=[sibling distance=35mm]
\tikzstyle{level 2}=[sibling distance=15mm]

\node{$\exists x\,\forall y\,\exists z\, P(x,y,z)$}
[grow=north]
child{node{$\forall y\,\exists z\, P(t_{3},y,z)$}  
child{node{$\exists z\, P(t_{3},a_{3},z)$} 
edge from parent node[left]{$a_{3}$}} 
edge from parent node[left]{$t_{3}$}}
child{node{$\forall y\,\exists z\, P(t_{2},y,z)$} 
child{node{$\exists z\, P(t_{2},a_{2},z)$} edge from parent node[left]{$a_{2}$}} 
edge from parent node[left]{$t_{2}$}}
child{node{$\forall y\,\exists z\, P(t_{1},y,z)$} 
child{node{$\exists z\, P(t_{1},a_{1},z)$} 
child{node{$ P(t_{1},a_{1},t_{4})$} edge from parent node[left]{$t_{4}$}} 
edge from parent node[left]{$a_{1}$}} 
edge from parent node[left]{$t_{1}$}};
\end{tikzpicture}$$
The representation of a proof as a two dimensional object, like an expansion tree, makes prominent space at the expense of time; the representation of a proof as a linear sequence, like in Hilbert's systems, gives priority to time at the cost of space. How to represent at once space and time? 

We first transform a cut-formula of a proof under focus into an expansion tree; this operation extracts the geometrical information hidden in the proof. We then rearrange the expansion tree in linear form, according to the order in which its nodes were introduced in the proof, obtaining a pointer sequence: this operation unfolds the temporal information coded in the proof. Pointer sequences then solve the problem of packing space and time information into a single object: the geometrical information is preserved by means of backwards pointers, while the temporal information is represented by the linear ordering of nodes!  
And finally something appears right before our eyes, something very clear, yet something we were completely blind to: the edges of the tree crossing each other. That is all we need to measure a proof's complexity.
Let us study an example.

The arena for the Coquand-style backtracking game which corresponds to the formula $\exists x\,\forall y\,\exists z\, P(x,y,z)$  comprises as set of moves and justification relation:

$$M\, =\, \{*\}\cup\{\mathsf{x}:=t\ |\ t \mbox{ is a closed term}\}\cup \{\mathsf{y}:=t\ |\ t \mbox{ is a closed term}\}\cup \{\mathsf{z}:=t\ |\ t \mbox{ is a closed term}\}$$
$$\vdash\ =\ \{(*, \mathsf{x}:=t)\ |\ t \mbox{ closed}\}\cup\{(\mathsf{x}:=t_{1}, \mathsf{y}:=t_{2}) \ |\ \mbox{ $t_{1}, t_{2}$ closed}\} \cup \{(\mathsf{y}:=t_{1}, \mathsf{z}:=t_{2}) \ |\ \mbox{ $t_{1}, t_{2}$ closed}\} $$

\noindent As remarked by Heijltjes \cite{Heijltjes}, the expansion tree above may  be seen as a strategy: the term labels $t_{1}, t_{2}, t_{3}, t_{4}$ are the witnesses to be played by Player for the corresponding formula of the tree, while the eigenvariable labels $a_{1}, a_{2}, a_{3}$ are to be replaced by the answers of Opponent. In order for the tree to be a correct specification of a strategy, Player must know in which order to play the moves, because the terms $t_{i}$ might contain eigenvariables $a_{j}$. The tree above is \textbf{correct}  if there is a total ordering of its edge labels $\ell_{1}, \ell_{2}, \ldots, \ell_{n}$ such that: 

\begin{enumerate}
\item if $i< j$ and $\ell_{i}$ and $\ell_{j}$ are on the same branch, then $\ell_{i}$ comes before $\ell_{j}$ in that branch.

\item for every $i, j$, if $\ell_{i}$ is an eigenvariable and $\ell_{j}$ a term containing it, $i<j$.
\end{enumerate}

\noindent An expansion tree extracted from a cut-free proof is correct by construction \cite{Heijltjes}. A correct ordering of the edge labels of the tree above could be: $t_{1}\, a_{1}\, t_{2}\, a_{2}\, t_{3}\, a_{3}\, t_{4}$, which would give the following strategy for Player: play first $t_{1}$ for the root formula, then wait for the answer $a_{1}$, then backtrack and play $t_{2}$ (which can contain at most $a_{1}$ as variable) for the root formula,  then wait for the answer $a_{2}$, then backtrack and play $t_{3}$ (which can contain at most $a_{1}, a_{2}$ as variables) for the root formula,  then wait for the answer $a_{3}$ and finally play $t_{4}$ (which can contain at most $a_{1}, a_{2}, a_{3}$ as variables).  Formally,  the strategy $\sigma$ for Player is the set of even plays $s$ such that $\pview{s}$ is an initial segment of a play of the following shape:
 
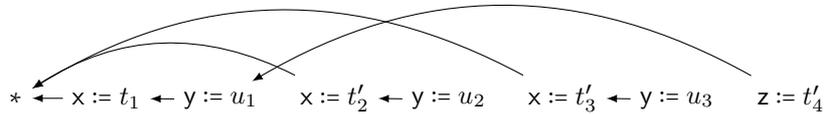
\vspace{-1ex}
 $$ \begin{tikzpicture}
 \node (0) at (-1.2,0.015)  {$*$};
\node (a) at (0,0.015) {$\mathsf{x}:= t_{1}$};
\node (b)  at (1.5,0)    {$\mathsf{y}:= u_{1}$};
 \node (c)  at (3,0.02)    {$ \mathsf{x}:= t_{2}'$};
 \node (d)  at (4.5,0)    {$\mathsf{y}:= u_{2}$};
\node (e)  at (6,0.02)    {$\mathsf{x}:= t_{3}'$};
\node (f)  at (7.5,0)    {$\mathsf{y}:= u_{3}$};
\node (g)  at (9,0.015)    {$ \mathsf{z}:=t_{4}'$}; 
\draw[-latex,right]  (b) edge (a);
 \draw[-latex,bend right]  (c) edge (0);
 \draw[-latex, right]  (a) edge (0);
\draw[-latex,right]  (d) edge (c);
\draw[-latex, bend right]  (e) edge (0);
\draw[-latex,right]  (f) edge (e);
\draw[-latex, bend right]  (g) edge (b); 
\end{tikzpicture}$$
 where $t_{2}'=t_{2}[u_{1}/a_{1}]$, $t_{3}'=t_{3}[u_{1}/a_{1}\ u_{2}/a_{2}]$,   $t_{4}'=t_{4}[u_{1}/a_{1}\ u_{2}/a_{2}\ u_{3}/a_{3}]$ and $u_{1}, u_{2}, u_{3}$ are terms representing the values for $a_{1}, a_{2}, a_{3}$ selected by Opponent. The definition of $\sigma$ might seem mysterious at first, but essentially there are no other choices: to determine his next move in a play $s$, Player must compute $\pview{s}$, which is the part of the history that concerns him, and then look at the corresponding initial segment of the linearized expansion tree.  It is immediate to verify that $\sigma$ is visible and bounded by $8$.  
 
 A strategy for Opponent, likewise, is given by any expansion tree for the involutive negation of the Player formula, for example: 
 $$
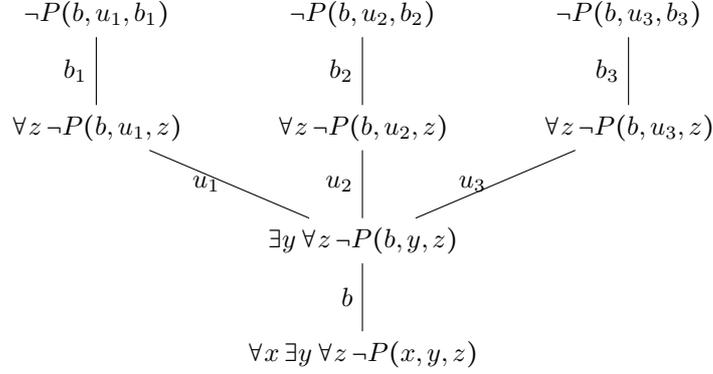
\begin{tikzpicture}
\tikzstyle{level 1}=[sibling distance=35mm]
\tikzstyle{level 2}=[sibling distance=35mm]

\node{$\forall x\, \exists y\,\forall z\, \lnot P(x,y,z)$}
[grow=north]
child{node{$\exists y\,\forall z\, \lnot P(b,y,z)$} 
child{node{$\forall z\, \lnot P(b,u_{3},z)$}  
child{node{$\lnot P(b,u_{3},b_{3})$} 
edge from parent node[left]{$b_{3}$}} 
edge from parent node[left]{$u_{3}$}}
child{node{$ \forall z\,\lnot P(b,u_{2},z)$} 
child{node{$\lnot P(b,u_{2},b_{2})$} edge from parent node[left]{$b_{2}$}} 
edge from parent node[left]{$u_{2}$}}
child{node{$\forall z\, \lnot P(b,u_{1},z)$} 
child{node{$\lnot P(b,u_{1},b_{1})$} edge from parent node[left]{$b_{1}$}} 
edge from parent node[left]{$u_{1}$}}
edge from parent node[left]{$b$}
};
\end{tikzpicture}$$
 A correct ordering of the edge labels of the tree above could be: $b\, u_{1}\, b_{1}\, u_{2}\, b_{2}\, u_{3}\, b_{3}$ and the corresponding  strategy $\tau$ for Opponent  is the set of odd plays $s$ such that $\oview{s}$ is an initial segment of a play of the following shape:
 
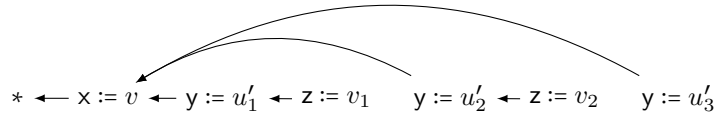
\vspace{-1ex}
$$ \begin{tikzpicture}
 \node (0) at (-1.2,0.0)  {$*$};
\node (a) at (0,0) {$\mathsf{x}:= v$};
\node (b)  at (1.5,0)    {$\mathsf{y}:= u_{1}'$};
 \node (c)  at (3,-0.01)    {$ \mathsf{z}:= v_{1}$};
 \node (d)  at (4.5,0)    {$\mathsf{y}:= u_{2}'$};
\node (e)  at (6,-0.01)    {$\mathsf{z}:= v_{2}$};
\node (f)  at (7.5,0)    {$\mathsf{y}:= u_{3}'$};
\draw[-latex,right]  (b) edge (a);
 \draw[-latex,bend right]  (d) edge (a);
 \draw[-latex, right]  (a) edge (0);
\draw[-latex,right]  (c) edge (b);
\draw[-latex, bend right]  (f) edge (a);
\draw[-latex,right]  (e) edge (d);
\end{tikzpicture}$$
 where $u_{1}'=u_{1}[v/b]$, $u_{2}'=u_{2}[v/b\ v_{1}/b_{1}]$, $u_{3}'=u_{3}[v/b\ v_{1}/b_{1}\ v_{2}/b_{2}]$ and  $v, v_{1}, v_{2}$ are terms representing the values for $b_{1}, b_{2}, b_{3}$ selected by Player. It is immediate to verify that $\tau$ is visible and bounded by $7$. 
 
The strategies $\sigma$ and $\tau$ are both eager to backtrack, but each one assumes that the other does not. Innocence \cite{HO} is the key tool to reconcile in a consistent manner these opposite desires: each strategy concentrates only on the current view to determine its answer, discarding the many backtracking moves that the other strategy may perform.  Computing the interactions in $\sigma\star\tau$ is thus very instructive for familiarizing with the concept of view and for seeing the already big combinatorial explosion that is generated. We leave to the readers the details of the computations. For simplicity of notation, we assume now that the terms $t_{1}, t_{2}, t_{3}, t_{4}, u_{1}, u_{2}, u_{3}$ do not contain variables, so that the substitutions have no effect. This simplification has no influence on the length of the interaction, which is totally independent of the shape of terms. An initial segment of the interaction between the two strategies then looks like this:
 $$ 
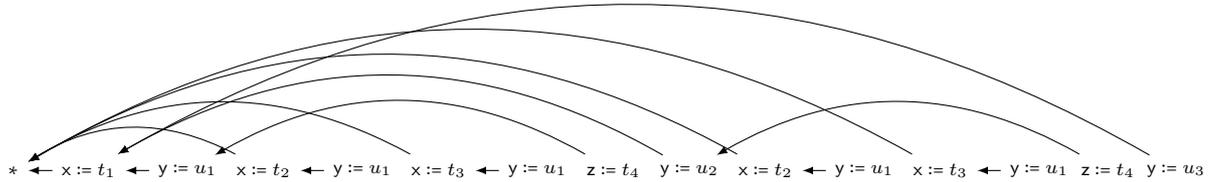
\begin{tikzpicture}[font=\scriptsize]
\node (0) at (0.4,0) {$*$};
\node (1) at (1.4,0) {$\mathsf{x}:=t_{1}$};
\node (2) at (2.7,0) {$\mathsf{y}:=u_{1}$};
\node (3) at (3.7,0) {$\mathsf{x}:=t_{2}$};
\node (4) at (5,0) {$\mathsf{y}:=u_{1}$};
\node (5) at (6,0) {$\mathsf{x}:=t_{3}$};
\node (6) at (7.3,0) {$\mathsf{y}:=u_{1}$};
\node (7) at (8.3,0) {$\mathsf{z}:=t_{4}$};
\node (8) at (9.3,0) {$\mathsf{y}:=u_{2}$};
\node (9) at (10.3,0) {$\mathsf{x}:=t_{2}$};
\node (10) at (11.6,0) {$\mathsf{y}:=u_{1}$};
\node (11) at (12.6,0) {$\mathsf{x}:=t_{3}$};
\node (12) at (13.9,0) {$\mathsf{y}:=u_{1}$};
\node (13) at (14.8,0) {$\mathsf{z}:=t_{4}$};
\node (14) at (15.7,0) {$\mathsf{y}:=u_{3}$};
  \draw[-latex,right]  (1) edge (0);
 \draw[-latex,right]  (2) edge (1);
 \draw[-latex,bend right]  (3) edge (0);
 \draw[-latex,right]  (4) edge (3);
 \draw[-latex,bend right]  (5) edge (0);
 \draw[-latex,right]  (6) edge (5);
  \draw[-latex,bend right]  (7) edge (2);
  \draw[-latex,bend right]  (8) edge (1);
    \draw[-latex,bend right]  (9) edge (0);
    \draw[-latex,right]  (10) edge (9);
    \draw[-latex,bend right]  (11) edge (0);
    \draw[-latex,right]  (12) edge (11);
    \draw[-latex, bend right]  (13) edge (8);
     \draw[-latex, bend right]  (14) edge (1);
\end{tikzpicture}
$$
and the interaction finishes with $\mathsf{x}:=t_{2}\ \mathsf{y}:=u_{1}\ \mathsf{x}:=t_{3}\ \mathsf{y}:=u_{1}\ \mathsf{z}:=t_{4}$. 

In general, a cut between two cut-free proofs of respectively $\Gamma, A$ and $\Gamma, A^{\bot}$  generates a number of interactions between two expansion trees, one for $A$ and one for the involutive negation $A^{\bot}$; the length of the cut-elimination process is then a simple function of the maximal length of these interactions. In this paper we do not investigate the exact connection between cut-elimination and strategy interaction,  the situation in first-order logic being slightly different from the situation in Arithmetic \cite{Herbelin}.  

\subsection{A Simple Property of Views}

The reader should make sure to be proficient with the concepts of view and visibility: several  proofs will require computations on views and checking visibility of plays. 
To ``warm up'', a simple property of views that we shall use very often is the following.

\begin{proposition}[Shape of Views]\label{prop-view}
Suppose $s= r\, \bullet\,r'\, \circ$ is a visible play over an arena $\mathbb{A}$ and there is an edge $(\bullet, \circ)$ in $s$. Then $s$ is of the form:

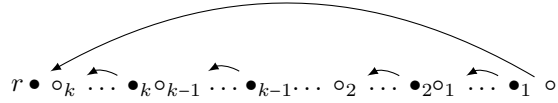
\vspace{-2ex}
$$\begin{tikzpicture}
\node  (z0) at (3.5,0.04) {$r\, \bullet$};
\node  (z) at (4,0) {$ \circ_{k}$};
\node  at (4.5,0) {$\ldots$};
\node (bk) at (5,0) {$\bullet_{k}$};
\node (ck-1) at (5.5,0) {$\circ_{k-1}$};
\node (z2) at (6.1,0) {$\ldots$};
\node (bk-1) at (6.7,0) {$\bullet_{k-1}$};
\node at (7.2, 0)   {$\ldots$};
\node (c2) at (7.7,0) {$\circ_{2}$};
\node (z2) at (8.2,0) {$\ldots$};
\node (b2) at (8.7,0) {$\bullet_{2}$};
\node  (c1) at (9,0) {$\circ_{1}$};
\node  at (9.5,0) {$\ldots$};
\node  (b1) at (10,0) {$\bullet_{1}$};
\node (d) at (10.4,0.03) {$\circ$};
\draw[-latex,bend right]  (d) edge (z0);
\draw[-latex,bend right]  (b1) edge (c1);
\draw[-latex,bend right]  (b2) edge (c2);
\draw[-latex,bend right]  (bk) edge (z);
\draw[-latex,bend right]  (bk-1) edge (ck-1);
\end{tikzpicture}$$
where, for $1\leq i \leq k$, $\bullet_{i+1}$ immediately precedes $\circ_{i}$ and there is an edge $(\circ_{i}, \bullet_{i})$ in $s$.

\end{proposition}
\begin{proof}
Let $\lambda(\circ)=p$. We can write $s$ as
$$\begin{tikzpicture}

\node at (2.7,0.04)   {$s=$};
\node at (3.1,0.04)   {$a$};
\node  (z) at (3.5,0) {$ \circ_{z}$};
\node  at (4,0) {$\ldots$};
\node (bk) at (4.5,0) {$\bullet_{z}$};
\node  at (5,0) {$\ldots$};
\node at (5.5, 0)   {$\ldots$};
\node (c2) at (6,0) {$\circ_{2}$};
\node (z2) at (6.5,0) {$\ldots$};
\node (b2) at (7,0) {$\bullet_{2}$};
\node  (c1) at (7.5,0) {$\circ_{1}$};
\node  at (8,0) {$\ldots$};
\node  (b1) at (8.5,0) {$\bullet_{1}$};
\node at (8.9,0.04)    {$\circ$};
\draw[-latex,bend right]  (b1) edge (c1);
\draw[-latex,bend right]  (b2) edge (c2);
\draw[-latex,bend right]  (bk) edge (z);
\end{tikzpicture}$$
  where, for $1\leq i \leq z$, $\bullet_{i+1}$ immediately precedes $\circ_{i}$ and there is an edge $(\circ_{i}, \bullet_{i})$ in $s$ and $a\in I$ or $a$ is the empty sequence. Therefore
the  $p$-view of $r\,\bullet\,r'$ is 
$$w=a\,\circ_{z}\,\bullet_{z}  \ldots\circ_{2}\,\bullet_{2}\circ_{1}\,\bullet_{1}$$ 
  Since $s$ is visible,  $\circ$ must point to a move in $w$, so either $\bullet=a$ or $\bullet=\bullet_{k}$, with $1\leq k< z$, which is what we  wanted to show. 
\end{proof}
Another well-known and simple-to-prove property of views (\cite{HO}, pp. 330--332) that we recall is that for every visible play $s$ and player $p$, the $p$-view of $s$ is itself a visible play.

\section{Berardi-de'Liguoro Backtracking Level}\label{sec-BacktrackingLev}

We now introduce the Berardi-de'Liguoro notion \cite{BerBack} of backtracking level of a strategy, which measures in a sophisticated and precise way the complexity of the backtracking moves made by the strategy. We shall have to extend the notion to arbitrary plays between strategies and investigate the resulting properties, thus continuing from where \cite{BerBack} stopped. 

Our first goal is to describe the concept of ``mind changing about a mind change'' or equivalently ``retraction of a retraction'', which will be formalized as the binary \emph{inactivation relation} $\inact$ between edges of a play. The idea, here, is that in a play any edge that points to a move back in the past of the play, implicitly erases, \emph{retracts}, all the moves between the two moves connected by the edge. A player act of pointing back to a past move, immediately implies that he is not satisfied with his former reactions to that move. It does not matter whether the reactions were ignoring the move altogether or answering it, the player has now changed his mind: either he has chosen to stop ignoring the move and answer it, or he decided to answer it differently. So, everything that was generated starting from his first former reaction no longer interests him, and he discards that part of the history. 
To continue the discussion we need the concept of crossing edges.

\begin{definition}[Crossing Edges]
Let $s=s_{0}\, s_{1}\, \ldots s_{n}$ be a play over an arena $\mathbb{A}$. We say that the edge $(s_{m_{1}}, s_{n_{1}})$ \textbf{crosses} the edge $(s_{m_{2}}, s_{n_{2}})$, if 
$$ m_{2} < m_{1}<n_{2}<n_{1}$$
which is represented in a picture as: 
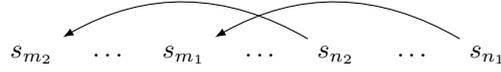
\vspace{-1ex}
$$\begin{tikzpicture}
\node (a) at (0,0) {$s_{m_{2}}$};
\node at (1,0)   {$\ldots$};
\node (b) at (2,0) {$s_{m_{1}}$};
\node at (3,0)   {$\ldots$};
\node (c) at (4,0) {$s_{n_{2}}$};
\node at (5,0)   {$\ldots$};
\node (d) at (6,0) {$s_{n_{1}}$};
\draw[-latex,bend right]  (d) edge (b);
\draw[-latex,bend right]  (c) edge (a);
\end{tikzpicture}$$
We write $e_{2} \cross e_{1}$ whenever an edge $e_{1}$ crosses an edge $e_{2}$, that is when $e_{2}$ \textbf{is crossed} by $e_{1}$. 
\end{definition}
Now, how to interpret a situation in which a player edge is crossed by another edge of the same player? This means that the player no longer agrees with his previous act of discarding the history between the ending points of the first edge. He wants to access something that was erased, he retracts his retraction. It is tempting to say that if an edge is crossed by another edge, then the latter always represents a mind change about the retraction made by the former and thus that the latter \emph{inactivates} the former. This definition would  not technically work: we shall see this in the discussion before Proposition \ref{prop-back1}, which should be already enough of a reason to stop considering the idea. However, such a definition would be just wrong in itself and inconsistent with our discussion. Consider the following situation:
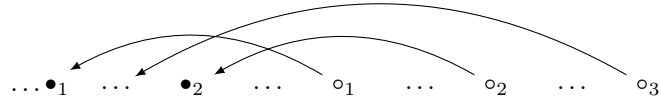
\vspace{-1ex}
$$\begin{tikzpicture}
\node (a) at (0,0) {$\ldots \bullet_{1}$};
\node (z1) at (1,0) {$\ldots$};
\node (b) at (2,0) {$\bullet_{2}$};
\node (z2) at (3,0) {$\ldots$};
\node (c) at (4,0) {$\circ_{1}$};
\node (z3) at (5,0) {$\ldots$};
\node (d) at (6,0) {$\circ_{2}$};
\node (z5) at (7,0) {$\ldots$};
\node (e) at (8,0) {$\circ_{3}$};
\draw[-latex,bend right]  (d) edge (b);
\draw[-latex,bend right]  (c) edge (a);
\draw[-latex,bend right]  (e) edge (z1);
\end{tikzpicture}$$
If we assume the second edge $\efrom{\circ_{2}}$ inactivates the first edge $\efrom{\circ_{1}}$, then the moves between $\bullet_{1}$ and $\circ_{1}$ may very well be accessible at the moment of the move $\circ_{3}$; in other words, since the second edge $\efrom{\circ_{2}}$ has already retracted the first retraction $\efrom{\circ_{1}}$, the third edge $\efrom{\circ_{3}}$ retracts nothing at all, even if it crosses the first. The point is that the first edge is \emph{active} at the moment the second edge is played, while it could be \emph{inactive}, i.e. not ``existing'', when the third edge is played. An edge is said to be inactive if  there is an active edge crossing it: a  real later retraction.  An edge is said to be active if all edges crossing it are inactive: no real retractions. One edge is inactivated by another one, if the former is active at the very moment in which the latter is created. Since in order to determine whether an edge is active or inactive  the future of the play is involved, the Berardi-de'Liguoro Definition \ref{def-inact} of activity  and the inactivation relation is by backward induction.

\begin{definition}[Active and Inactive Edges of a Play, Inactivation]\label{def-inact} Let $s=s_{0}\, s_{1}\, \ldots s_{n}$ be a play over an arena $\mathbb{A}$. 
\begin{itemize}
\item We define by backward induction, for all $0\leq i\leq n$, whether the edge from $s_{i}$ is \textbf{active} or \textbf{inactive in} $s$.
\begin{enumerate}
\item $\efrom{s_{n}}$ is active.
\item $\efrom{s_{i}}$, with $i<n$, is inactive if there is a $j>i$ such that $\efrom{s_{i}}\cross \efrom{s_{j}}$ and $\efrom{s_{j}}$ is active. 
\item $\efrom{s_{i}}$, with $i<n$, is active if for all  $j>i$, if  $\efrom{s_{i}}\cross \efrom{s_{j}}$, then $\efrom{s_{j}}$ is inactive. 
\end{enumerate}
\item We say that $\efrom{s_{i}}$ \textbf{inactivates} $\efrom{s_{j}}$ \textbf{in} $s$ or that $\efrom{s_{j}}$ \textbf{is inactivated by} $\efrom{s_{i}}$ \textbf{in} $s$, and write $s_{j}\inact s_{i}$ in $s$, if  $\efrom{s_{j}}$ is active in $s_{0}\, s_{1}\ldots \, s_{i-1}$ and $\efrom{s_{j}}\cross \efrom{s_{i}}$. \end{itemize}
\end{definition}
The combinatorial power of the definition lies in its logical complexity: the condition determining whether a edge is active or inactive in a play is a formula with possibly as many alternating quantifiers as the length of the play. We remark that the concept of inactivation is always relative to a specified play and has a temporal connotation; we will always make sure to declare with respect to what play the inactivation relation is considered. As an example of active and inactive edges, let us consider the following visible play:
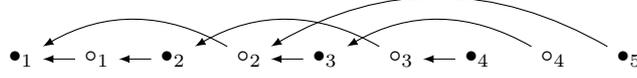
\vspace{-1ex}
 $$\begin{tikzpicture}
 \node (x) at (0,0) {$\bullet_{1}$};
\node (a) at (1,0) {$\circ_{1}$};
\node (b) at (2,0) {$\bullet_{2}$};
\node (c) at (3,0) {$\circ_{2}$};
\node (d) at (4,0) {$\bullet_{3}$};
\node (e) at (5,0) {$\circ_{3}$};
\node (f)  at (6, 0)  {$\bullet_{4}$};
\node (g)  at (7, 0)  {$\circ_{4}$};
\node (h)  at (8, 0)  {$\bullet_{5}$};
\draw[-latex, right]  (b) edge (a);
\draw[-latex, right]  (a) edge (x);
\draw[-latex,bend right]  (c) edge (x);
\draw[-latex, right]  (d) edge (c);
\draw[-latex,bend right]  (e) edge (b);
\draw[-latex, right]  (f) edge (e);
\draw[-latex,bend right]  (g) edge (d);
\draw[-latex,bend right]  (h) edge (c);
\end{tikzpicture}$$
The edges $\efrom{\bullet_{5}}$ are $\efrom{\circ_{4}}$ active in the whole play, thus the edge $\efrom{\circ_{3}}$ is inactive and the edge $\efrom{\circ_{2}}$ is active in the whole play. Moreover, $\efrom{\circ_{3}}$ is active in the initial segment of the play ending with the move $\bullet_{4}$, therefore $\efrom{\circ_{3}}\inact \efrom{\circ_{4}}$. However, it is not the case that $\efrom{\circ_{3}}\inact \efrom{\bullet_{5}}$, for $\efrom{\circ_{3}}$ is inactive in the initial segment of the play ending with the move $\circ_{4}$. 

The time has come to define the backtracking level of a move, play and strategy. The idea is to consider the chain of mind changes that originates from a move. A move might be a simple retraction, and thus have backtracking level $1$, or it might be a retraction of a retraction, and thus have backtracking level $2$, or it might be a retraction of a retraction of a retraction, and thus have backtracking level $3$. In general, one measures the length of a maximal chain of such retractions of retractions. More formally, given a move $a$, if $a$ does not retract anything because it answers the immediately previous move, then it is declared to have backtracking level $0$. Otherwise, one defines the backtracking level of $a$ as the length $k+1$ of a maximal chain of the shape $e_{1}\inact e_{2}\inact \ldots\inact e_{k}\inact\efrom{a}$; in particular, if there is no $e$ such that $e\inact a$, the backtracking level of $a$ is $1$. The backtracking level of a play is the maximum backtracking level of its moves, while the backtracking level of a bounded strategy is the maximum among the backtracking levels of the views of the plays in it.

\begin{definition}[Backtracking level of a Move, Play and Strategy]\label{def-backlevel}
Let $s=s_{0}\, s_{1}\, \ldots s_{n}$ be a play over an arena $\mathbb{A}$. 
\begin{itemize}
\item We define by induction, for all $0\leq i\leq n$, the \textbf{backtracking level in $s$ of the move} $s_{i}$ ($\blev(s_{i})$) as following. 

If  $(s_{i-1}, s_{i})$ is an edge of $s$ or ${i}=0$, then $\blev(s_{i})=0$.

If  $(s_{i-1}, s_{i})$ is not an edge of $s$ and ${i}>0$, then  $\blev(s_{i})=k+1$, where $k$ is the length of the longest chain $e_{1}, e_{2}, \ldots, e_{k}$ of edges of $s$ such that in $s$
$$e_{1}\inact e_{2}\inact \ldots\inact e_{k}\inact\efrom{s_{i}}$$
\item We define the \textbf{backtracking level of the play $s$} as: $$\blev(s)=\mathsf{max}\{\blev(s_{i})\ | \ 0\leq i\leq n\}$$
We define the \textbf{backtracking level of the player $p$ in the play $s$} as: $$p\blev(s)=\mathsf{max}\{\blev(s_{i})\ | \ 0\leq i\leq n\land \lambda(s_{i})=p\}$$

\item For any bounded Player strategy $\sigma$, we define its backtracking level as: $$\blev(\sigma)=\mathsf{max}\{\blev(\pview{s})\ | \ s\in \sigma\}$$
\item For any bounded Opponent strategy $\sigma$, we define its backtracking level as: $$\blev(\sigma)=\mathsf{max}\{\blev(\oview{s})\ | \ s\in \sigma\}$$
\item If $\blev(\sigma)=n$, we say  that $\sigma$ is a \textbf{$n$-backtracking strategy}.\\
\end{itemize}
\end{definition}
Let us consider the following example:
 $$
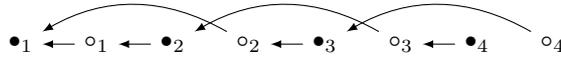
\begin{tikzpicture}
 \node (x) at (0,0) {$\bullet_{1}$};
\node (a) at (1,0) {$\circ_{1}$};
\node (b) at (2,0) {$\bullet_{2}$};
\node (c) at (3,0) {$\circ_{2}$};
\node (d) at (4,0) {$\bullet_{3}$};
\node (e) at (5,0) {$\circ_{3}$};
\node (f)  at (6, 0)  {$\bullet_{4}$};
\node (g)  at (7, 0)  {$\circ_{4}$};
\draw[-latex, right]  (b) edge (a);
\draw[-latex, right]  (a) edge (x);
\draw[-latex,bend right]  (c) edge (x);
\draw[-latex, right]  (d) edge (c);
\draw[-latex,bend right]  (e) edge (b);
\draw[-latex, right]  (f) edge (e);
\draw[-latex,bend right]  (g) edge (d);
\end{tikzpicture}$$
The move $\circ_{1}$ is of backtracking level $0$, $\circ_{2}$ of backtracking level $1$, $\circ_{3}$ of backtracking level $2$ and $\circ_{4}$ of backtracking level $3$.  But in the following play
$$
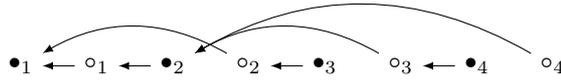
\begin{tikzpicture}
 \node (x) at (0,0) {$\bullet_{1}$};
\node (a) at (1,0) {$\circ_{1}$};
\node (b) at (2,0) {$\bullet_{2}$};
\node (c) at (3,0) {$\circ_{2}$};
\node (d) at (4,0) {$\bullet_{3}$};
\node (e) at (5,0) {$\circ_{3}$};
\node (f)  at (6, 0)  {$\bullet_{4}$};
\node (g)  at (7, 0)  {$\circ_{4}$};
\draw[-latex, right]  (b) edge (a);
\draw[-latex, right]  (a) edge (x);
\draw[-latex,bend right]  (c) edge (x);
\draw[-latex, right]  (d) edge (c);
\draw[-latex,bend right]  (e) edge (b);
\draw[-latex, right]  (f) edge (e);
\draw[-latex,bend right]  (g) edge (b);
\end{tikzpicture}$$
$\circ_{4}$ is of backtracking level $1$, since in the initial segment of the play terminating with $\bullet_{4}$ the edge $\efrom{\circ_{3}}$ is active and thus $\efrom{\circ_{2}}$ inactive.

\section{Ideas of the Proof}\label{sec-proofidea}
We now have all the tools needed first to enter  and then survive the agitated waters of complexity analysis. In order not to lose ourselves in the details before reaching the final destination, we  trace a map with the principal objectives to achieve and their role in our proof.
The lowest layer of our proof strategy is built upon some reflections about two crucial features of Coquand's proof of debate termination \cite{Coquand}, so it is wise to  look back at it first.
\subsection{Back to Coquand's proof: Uncrossed Edges and the Trimming Operation}
Coquand proved in \cite{Coquand} that in any arena of finite depth debates between well-founded strategies always terminate. The proof's main idea is to consider the so called \emph{orphan} moves, which are those moves that during the debate under focus are not directly answered by any later move. 
One can easily prove than in any visible play $s$, the edges from such moves can never be crossed. If we cut away from $s$  all the moves that are in the interior of an edge from some orphan move, removing also the edge's ending points, we obtain again a visible play $s'$. This play takes place by construction in an arena of depth strictly less than the original one, because all the moves of maximal depth are removed from the original play $s$. Therefore one can argue by induction that $s'$ is of finite length and then  easily calculate the length of the original play $s$. Interestingly, if we ignore all the other edges of the play, the edges from orphan moves can be seen from outside as $1$-backtracking moves, because they never cross each other. Thus if we watch the play flowing from left to right like a sequence of photograms and, as soon as an edge from an orphan move appears, we erase all the moves in its interior, we see a $1$-backtracking game \cite{BerCoqHay}, for example of the form:
\vspace{-2ex}
$$\circ_{1};\ \ \ \circ_{1}\, \circ_{2};\ \ \  \circ_{1}\, \circ_{2}\, \circ_{3};\ \ \  \circ_{1}\, \circ_{2}\, \circ_{3}\, \circ_{4};\ \ \  \circ_{1};\ \ \  \circ_{1}\, \circ_{5}; \ \ \  \circ_{1}\, \circ_{5}\, \circ_{6}; \ \ \  \circ_{1}\, \circ_{5}\, \circ_{6}\, \circ_{7}; \ \ \  \circ_{1}\, \circ_{5};\ \ \ldots$$
That is, each sequence of moves is either an extension or an initial segment of the previous sequence. These $1$-backtracking plays are easy to bound, so the proof goes through. 

\subsection{The Max-Min Choice Technique}

If one thinks carefully about the argument above, one can identify two crucial features: \\
\begin{enumerate}
\item a choice of edges with the property that they cannot be crossed by any other edge;
\item a trimming operation that cuts away some unwanted part of the original play -- the moves in the intervals defined by the previously chosen edges -- to produce  a play in an arena of lower depth. 
\end{enumerate}
The problem is that the induction that one winds up to employ is on the depth of the arena. Thus Coquand's argument yields again the old, well-known tower of exponentials, as high as the depth of the arena. If we want to obtain finer results, we need a more sophisticated induction. The idea is to choose differently the edges to remove and thus make  use of a different trimming operation:

\begin{enumerate}
\item we choose as edges to remove those from the moves of \emph{maximal} backtracking level among those made by the player with \emph{minimal} backtracking level in the play;
\item the trimming operation cuts away some unwanted part of the original play -- those moves contained in the interior of the previously chosen edges -- to produce a play in which the previously chosen player has backtracking level strictly less than in the original play.
\end{enumerate}
These concepts inspire the following definition.

\begin{definition}[Real Backtracking Level of a Play, Principal Player]
Let $s$ be a visible play.
\begin{itemize}
\item Let $n$ be the backtracking level of Player in $s$ and $m$ the backtracking level of Opponent in $s$. We define the \textbf{real backtracking level of} $s$ as $\mathsf{min}(n,m)$. If $\mathsf{min}(n,m)=n$, then Player is called the \textbf{principal player of} $s$; otherwise, Opponent is the \textbf{principal player of} $s$.
 \end{itemize}
\end{definition}
According to this terminology, the proof strategy described above identifies the moves by the principal player of a play that have backtracking level equal to the real backtracking level of the play. We speak of \emph{real} backtracking level, because thanks to the deep symmetry between the two players of the game, one can choose to look at a play \emph{with the eyes} of the principal player. If one looks instead at the very same play from the perspective of the other player, one sees an \emph{apparent} backtracking level higher than the real one and thus the play looks, wrongly, more complex than really is. 

With this approach, we can make a refined induction on the backtracking level of the principal player of the play, thus obtaining a tower of exponential possibly much lower than the old, well-known one. However, we shall need a lot of technical preparatory study to make our proof strategy succeed. The fact that one can focus just on the principal player of a play is surprising. Actually, it may seem a stroke of luck that such an argument can work, because it faces many obstructions that we are going to describe in a moment. As a matter of fact, we discovered our results by means of a more intuitive geometrical proof; as it was much more difficult to formalize, we discarded it and turned to the present approach. 

\subsection{Obstacles}

The \textbf{first obstacle} to address is the following. Given a visible play, consider the edges from the moves of \emph{maximal} backtracking level among those made by the principal player $p$, that is the one with \emph{minimal} backtracking level in the play: can these edges be crossed? It is not trivial that they do not -- let alone that they can assume the role that orphan moves have in Coquand's proof. Since the other player $\overline{p}$ may have backtracking level greater than that of $p$, the edge from some $\overline{p}$-move could in principle cross and inactivate the edge from a $p$-move of maximal backtracking level \emph{among those of} $p$. Therefore, we need to prove that this cannot happen: it will be a task of  Section \S\ref{sec-maxmin}. 

The \textbf{second obstacle} to address is: if two strategies interact together, what is the backtracking level of the resulting play? Could it be, for example, that if a Player strategy has backtracking level $2$ and an Opponent strategy has backtracking level $5$, then in the play resulting from their interaction Player has backtracking level $5$ or even $5+2$, due to the interference of the Opponent moves? In Section \S\ref{sec-maxmin}, we rule out this possibility and show that in the resulting play the backtracking level of Player is still $2$ and that of Opponent still $5$. 
This is not obvious, because chains of crossing edges of length $5+2$ could pop up at any moment in the play. 

The \textbf{third obstacle} to address is whether the Trimming operation really transforms a visible play, first of all, into a visible play, secondly, into a play in which the principal player of the initial play has backtracking level strictly smaller. In Coquand's proof the Trimming operation is easily shown to transform visible plays into visible plays; in our proof we shall have to investigate in depth the shape of views to check that it is the case. The study will be carried out in Section \S\ref{sec-trimm}.

There is a last tricky \textbf{fourth obstacle} to address. What if a strategy has backtracking level greater than the depth of the arena? This is a curious situation, but the answer is essentially: such a  strategy is \emph{badly defined}. Section \S\ref{sec-gadget} will settle the matter: it will be shown that inserting some dummy moves in a strategy, that form a sort of switch-on/switch-off graph-theoretic gadget, one can get a strategy of backtracking level less than or equal to the depth of the arena. The same can be done with an arbitrary play.  This provides conclusive evidence that in principle, a strategy should never have backtracking level greater than the depth of the correspondent arena: if it does, it has been constructed in a careless manner.

\section{Maximal Edges and The Max-Min Theorem of Backtracking Levels Interaction}\label{sec-maxmin}
In this Section  we establish  several key properties of the inactivation relation between edges, removing the first two obstacles in our way to the main results.  Namely, we shall see that the edges from the moves of maximal backtracking level, among those  of one player,  cannot be crossed, let alone inactivated. Moreover, we will show that any interaction between  two strategies of backtracking level respectively $n$ and $m$ is a play of backtracking level less than or equal to the maximum among $n$ and $m$ and of real backtracking level less than or equal to the minimum among $n$ and $m$. We shall call this property the Max-Min Theorem of Backtracking Levels Interaction.  

We start off by proving that only the very same player that made a still active backtracking can retract it with another backtracking: in other terms, he alone has the power to recover the moves that he himself ``erased''. Unless he decides that those moves have to be reconsidered, the other player will never be able to see them again.  More precisely, if $e$ is a $p$-edge inactivated by the edge $e'$ in  a visible play $s$, then $e'$ is a $p$-edge as well. Moreover, the inactivation relation in $s$ reflects exactly the inactivation relation in the views of $s$: if two $p$-edges are in the inactivation relation in $s$ and the second is in the $p$-view of $s$, then  they are in the inactivation relation also in the $p$-view of $s$, and vice versa.  These properties are due to the special nature of the inactivation relation and are absolutely essential for the rest of the paper, so we check them carefully. Already the first would not hold with other naive definitions of $\inact$, namely if $\inact$ had been  taken to be equal to the crossing relation $\cross$. For example, whenever in some visible play $\efrom{\circ_{1}}\cross \efrom{\circ_{2}}$, the move $\bullet$ immediately following the move $\circ_{2}$ is of opposite polarity and the corresponding player can see in the interior of the first edge and thus point to a move in it, enabling the crossing relation $\efrom{\circ_{1}}\cross \efrom{\bullet}$:
$$\begin{tikzpicture}
\node (a) at (0.8,0) {$\bullet_{1}$};
\node (z1) at (1.4,0) {$\circ$};
\node (b) at (2,0) {$\bullet_{2}$};
\node (z2) at (3,0) {$\ldots$};
\node (c) at (4,0) {$\circ_{1}$};
\node (z3) at (5,0) {$\ldots$};
\node (d) at (6,0) {$\circ_{2}$};
\node (e) at (6.6,0) {$\bullet$};
\draw[-latex,bend right]  (d) edge (b);
\draw[-latex,bend right]  (c) edge (a);
\draw[-latex,bend right]  (e) edge (z1);
\end{tikzpicture}$$
From now on we fix an arena $\mathbb{A}=(M, \vdash, \lambda, I)$ and we shall consider only plays over $\mathbb{A}$.

\begin{proposition}[Preservation of Inactivations]\label{prop-back1}
Suppose $s$ is a visible play, $e$ is a $p$-edge of $s$ and $e'=(\bullet, \circ)$ is another edge of $s$. 
\begin{enumerate}
\item  If  $e\inact e'$ in $s$, then $e'$ is a $p$-edge.
\item Assume $e'$ is in the $p$-view $w$ of $s$. Then $e \inact e'$ in $s$ if and only if $e\inact e'$ in  $w$.
\end{enumerate}

 \end{proposition}
\begin{proof}
Moves by $p$ and $\overline{p}$ will be represented by different colors: for the moment we do not assume that white is for $p$, nor that black is for $p$. We need first to establish some facts.  By hypothesis $s=s'\, \bullet\, s'' \circ\, s'''$; by Proposition \ref{prop-view}, $s''=\circ_{k}\ldots\bullet_{k}\circ_{k-1}\ldots\bullet_{k-1}  \ldots\circ_{2}\ldots\bullet_{2}\,\circ_{1}\ldots\bullet_{1}$
and
 \vspace{-2ex}
$$\begin{tikzpicture}
\node at (2.5,0.04) {$s \ =\ $};
\node  at (3.2,0.1) {$s'$};
\node  (z0) at (3.5,0.02) {$\bullet$};
\node  (z) at (4,0) {$ \circ_{k}$};
\node  at (4.5,0) {$\ldots$};
\node (bk) at (5,0) {$\bullet_{k}$};
\node (ck-1) at (5.5,0) {$\circ_{k-1}$};
\node (z2) at (6.1,0) {$\ldots$};
\node (bk-1) at (6.7,0) {$\bullet_{k-1}$};
\node at (7,1.4)   {$e'$};
\node at (7.2, 0)   {$\ldots$};
\node (c2) at (7.7,0) {$\circ_{2}$};
\node (z2) at (8.2,0) {$\ldots$};
\node (b2) at (8.7,0) {$\bullet_{2}$};
\node  (c1) at (9,0) {$\circ_{1}$};
\node  at (9.5,0) {$\ldots$};
\node  (b1) at (10,0) {$\bullet_{1}$};
\node (d) at (10.5,0.02) {$\circ$};
\node (e) at (10.9,0.1) {$s'''$};
\draw[-latex,bend right]  (d) edge (z0);
\draw[-latex,bend right]  (b1) edge (c1);
\draw[-latex,bend right]  (b2) edge (c2);
\draw[-latex,bend right]  (bk) edge (z);
\draw[-latex,bend right]  (bk-1) edge (ck-1);
\end{tikzpicture}$$
We argue, by induction on $i$, that  each edge $(\circ_{i}, \bullet_{i})$, with $1\leq i\leq k$ is active in $s'\,\bullet\, s''$.\\
 For $i=1$, that holds, since $(\circ_{1}, \bullet_{1})$ is crossed by no edge in $s'\,\bullet\, s''$, being it the last edge. \\
So, suppose $i>1$. Then \vspace{-2ex}
$$\begin{tikzpicture}
\node at (2.5,0.04) {$s \ =\ $};
\node  at (3.2,0.1) {$s'$};
\node  (z0) at (3.5,0.02) {$\bullet$};
\node  (z) at (4,0) {$ \circ_{k}$};
\node  at (4.5,0) {$\ldots$};
\node (bk) at (5,0) {$\bullet_{k}$};
\node  at (5.5,0) {$\ldots$};
\node (ck-1) at (6,0) {$\circ_{i}$};
\node (z2) at (6.5,0) {$\ldots$};
\node (bk-1) at (7,0) {$\bullet_{i}$};
\node at (7,1.4)   {$e'$};
\node at (7.5, 0)   {$\ldots$};
\node (c2) at (8,0) {$\circ_{2}$};
\node (z2) at (8.5,0) {$\ldots$};
\node (b2) at (9,0) {$\bullet_{2}$};
\node  (c1) at (9.3,0) {$\circ_{1}$};
\node  at (9.8,0) {$\ldots$};
\node  (b1) at (10.3,0) {$\bullet_{1}$};
\node (d) at (10.8,0.02) {$\circ$};
\node (e) at (11.3,0.1) {$s'''$};
\draw[-latex,bend right]  (d) edge (z0);
\draw[-latex,bend right]  (b1) edge (c1);
\draw[-latex,bend right]  (b2) edge (c2);
\draw[-latex,bend right]  (bk) edge (z);
\draw[-latex,bend right]  (bk-1) edge (ck-1);
\end{tikzpicture}$$
By induction hypothesis, for any $j<i$, $(\circ_{j}, \bullet_{j})$ is active in $s'\,\bullet\, s''$; therefore every edge from every move $a$ lying between the moves $\circ_{j}$ and $\bullet_{j}$  and such that $\efrom{a}\cross (\circ_{j}, \bullet_{j})$, is by Definition \ref{def-inact} inactive in $s'\,\bullet\, s''$. Then, the only active edges that could cross $(\circ_{i}, \bullet_{i})$ are those of the form $\efrom{\circ_{j}}$, for some $j<i$; but that cannot happen, because by visibility of $s$, $\circ_{j}$ must point to an element  in the $\lambda(\circ_{j})$-view $v$ of the play
$$s'\,\bullet\,\circ_{k}\ldots\bullet_{k}\ldots\circ_{i}\ldots\bullet_{i}  \ldots\circ_{j+1}\ldots\bullet_{j+1}$$
and since $v=v'\,\bullet\,\circ_{k}\,\bullet_{k}\ldots \circ_{i}\,\bullet_{i}\ldots\circ_{j+1}\,\bullet_{j+1}$, it  is not the case that $(\circ_{i},\bullet_{i})\cross \efrom{\circ_{j}}$ in $s$.\\
\begin{enumerate}
\item

Now, assume $e\inact e'=(\bullet,\circ)$ in $s$. By Definition \ref{def-inact} of the inactivation relation, $e$ is active in $s'\,\bullet\, s''$ and $e\cross e'$. For every move $a$ lying between $\circ_{i}$ and $\bullet_{i}$ for some $1\leq i\leq k$ and such that $\efrom{a}\cross (\bullet,\circ)$, we deduce that $\efrom{a}$ is inactive in $s'\,\bullet\, s''$, since $\efrom{a}\cross (\circ_{i}, \bullet_{i})$ and $(\circ_{i}, \bullet_{i})$ is active in $s'\,\bullet\, s''$.  Therefore, it must be the case that $e=\efrom{\circ_{l}}$, for some $1\leq l\leq k$, which establishes that  white is the color for $p$ moves and that $e'$ is a $p$-edge.\\

\item  Let $w$ be the $p$-view of $s$. If  $e'$ is not a $p$-edge, it is impossible both that $e\inact e'$ in $s$, by 1., and that $e\inact e'$ in $w$, because every $\overline{p}$ move in $w$ points to the immediately preceding move; thus the thesis is trivial. We can hence assume that $e'$ is a $p$-edge. Since $e'$ is in $w$, for some sequences $r, r'$ we have 
$$w=r\,\bullet\,\circ_{k}\,\bullet_{k}\circ_{k-1}\,\bullet_{k-1}  \ldots\circ_{2}\,\bullet_{2}\,\circ_{1}\,\bullet_{1}\circ\, r'$$
and for $1 \leq i \leq k$, the edge $(\circ_{i}, \bullet_{i})$ is in $w$. Let
$$w'=r\,\bullet\,\circ_{k}\,\bullet_{k}\circ_{k-1}\,\bullet_{k-1}  \ldots\circ_{2}\,\bullet_{2}\,\circ_{1}\,\bullet_{1}$$ 
We argue, by induction on $i$, that each edge $\efrom{\circ_{i}}$, with $1\leq i\leq k$, is active  in $w'$ if it is active in $s'\,\bullet\, s''$ and is inactive in $w'$ if it is inactive in $s'\,\bullet\, s''$.\\
 For $i=1$, that holds, since $\efrom{\circ_{1}}$ is not crossed by any edge in $w'$ and is thus active in $w'$; moreover, $\efrom{\circ_{1}}$ is active in $s'\,\bullet\, s''$, because $(\circ_{1},\bullet_{1})$ is active in $s'\,\bullet\, s''$ and therefore whenever  $\efrom{\circ_{1}}\cross d$ with $d$ in $s'\,\bullet\, s''$, surely $d\neq  (\circ_{1},\bullet_{1})$, so  one has $d\cross (\circ_{1},\bullet_{1})$ and so $d$ is inactive in $s'\,\bullet\, s''$. \\
 So, suppose $i>1$. 
 \begin{itemize}
 \item If $\efrom{\circ_{i}}$ is active in $s'\,\bullet\, s''$, all the edges $\efrom{\circ_{j}}$, such that $j<i$ and $\efrom{\circ_{i}}\cross \efrom{\circ_{j}}$, are inactive in $s'\,\bullet\, s''$; by induction hypothesis, they are inactive in $w'$ as well, and since there is  no other edge crossing $\efrom{\circ_{i}}$ in $w'$, we have that $\efrom{\circ_{i}}$ is active in $w'$. 
 \item If $\efrom{\circ_{i}}$ is inactive in $s'\,\bullet\, s''$, then there is an active edge $d$ in $s'\,\bullet\, s''$ such that $\efrom{\circ_{i}}\cross d$; but $d$ must be equal to $\efrom{\circ_{j}}$, for some $j<i$, since all the other edges crossing $\efrom{\circ_{i}}$ are inactive, because they are crossed by some active edge $(\circ_{l}, \bullet_{l})$; by induction hypothesis, $\efrom{\circ_{j}}$ is active in $w'$, moreover $\efrom{\circ_{i}}\cross \efrom{\circ_{j}}$ in $w'$ and so $\efrom{\circ_{i}}$ is inactive in $w'$.  
\end{itemize}
Finally, let us prove the thesis. Assume $e\inact e'$ in $s$. By the proof of 1., $e=\efrom{\circ_{l}}$, with $1\leq l\leq k$, and $e$ is active in $s'\,\bullet\, s''$ and thus in $w'$. Therefore $e\inact e'$ in $w$. Assume now  $e\inact e'$ in $w$. Then $e=\efrom{\circ_{l}}$, with $1\leq l\leq k$, and $e$ is active in $w$ and thus in $s'\,\bullet\, s''$. Therefore $e\inact e'$ in $s$.
\end{enumerate}
\end{proof}
Proposition \ref{prop-back2} below provides a simple way to compute the backtracking level of a play: it is enough to determine the backtracking level of the views of its initial segments. This is possible because the inactivation chains of a visible play are exactly those already contained in the views of the play's initial segments. 
\begin{proposition}[Preservation of Inactivation Chains]\label{prop-back2}
Let $s=s_{0}\, s_{1}\ldots s_{k}$ be a visible play.  Let  $e_{1}, e_{2},\ldots, e_{n}$ be edges of $s$, assume $e_{n}=\efrom{s_{i}}$ is a $p$-edge, for some player $p$, and $w$ is the $p$-view of $s_{0}\, s_{1}\ldots s_{i}$. Then
$$e_{1}\inact e_{2}\inact \ldots\inact e_{n}\mbox{ in } s \iff e_{1}\inact e_{2}\inact \ldots\inact e_{n}\mbox{ in } w$$

 \end{proposition}
\begin{proof}  $\Longrightarrow$)
Assume $e_{1}\inact e_{2}\inact \ldots\inact e_{n}\mbox{ in } s$. By definition of view and visibility of $s$, $e_n$ belongs to the $p$-view $w$ of $s_{0}\,s_{1}\ldots s_{i}$. By repeated application of Proposition \ref{prop-back1}, $e_{n-1}\inact e_{n}$ in $w$, $e_{n-2}\inact e_{n-1}$ in $w$, $\ldots$, $e_{1}\inact e_{2}$ in $w$, which is the thesis. \\
$\Longleftarrow$) As in the other direction.
\end{proof}
We prove now the key property of maximal backtracking moves and of the edges from them, which states that if a player $p$ of backtracking level $n$ in some visible play has made a move of  backtracking level $n$, then the edge from that move cannot be crossed by any other edge, let alone inactivated. These are the edges that will play the role that the edges from unanswered moves played in Coquand's termination proof \cite{Coquand}.

\begin{proposition}[Maximal Backtracking Edges]\label{prop-back3}
Suppose $s=s_{0}\, s_{1}\ldots s_{k}$ is a visible play  and $p$ a player of backtracking level $n$ in $s$. Assume $\lambda(s_{i})=p$ and $\blev(s_{i})=n$.  Then there is no $j$ such that $\efrom{s_{i}}\cross \efrom{s_{j}}$ in $s$.
\end{proposition}
\begin{proof}
Let us consider, for the sake of contradiction, the smallest $j$ such that $\efrom{s_{i}}\cross \efrom{s_{j}}$ in $s$. By choice of $j$,  $\efrom{s_{i}}$ is not crossed by any edge in $s_{0}\,s_{1}\ldots s_{j-1}$, therefore $\efrom{s_{i}}$ is active in $s_{0}\,s_{1}\ldots s_{j-1}$ and so $\efrom{s_{i}}\inact \efrom{s_{j}}$ in $s$.  By Proposition \ref{prop-back1}, $\efrom{s_{j}}$ is a $p$-edge. Since $\blev(s_{i})= n$, there is in $s$ an inactivation chain 
$$e_{1}\inact e_{2}\inact \ldots\inact e_{n-1}\inact \efrom{s_{i}}\inact \efrom{s_{j}}$$ But this means that $s_{j}$ and thus $p$ are of backtracking level $n+1$ in $s$, which contradicts the hypothesis. 
\end{proof}

The next Theorem is the first major result of this paper: if two strategies respectively of backtracking level $n$ and $m$ play one against the other, they generate a visible play of backtracking level equal to the maximum among $n$ and $m$. However, the strategy with backtracking level equal to the minimum among $n$ and $m$ will maintain its backtracking  level also in the generated play. Therefore, the real backtracking level of the generated play is equal to the minimum among $n$ and $m$.

\begin{theorem}[Max-Min Property of Backtracking Levels Interaction]\label{thm-maxmin}\mbox
 Let $\sigma$ be a Player $n$-backtracking strategy, $\tau$  an Opponent $m$-backtracking strategy and assume $s\in\sigma\star\tau$. Then $s$ is a visible play of backtracking level less than or equal to $\mathsf{max}(n,m)$ and of real backtracking level less than or equal to $\mathsf{min}(n,m)$.

\end{theorem}
\begin{proof}
 In order to prove that $s$ is visible, suppose that $s=s'\, a$. Assume $\lambda(a)=P$. Since by Definition \ref{def-stratvisint} of interaction, $s'\,a \in \sigma$ and $s'\in \tau$, and $\sigma$ and $\tau$ are closed by even and odd prefixes, every prefix of $s$ belongs to $\sigma$ or $\tau$. Since $\sigma$ and $\tau$ are visible, we are done. If $\lambda(a)=O$, exchange $\sigma$ and $\tau$ in the previous argument. \\ 
In order to prove the thesis, it is enough to show that Player and Opponent have backtracking level in $s$ less than or equal to respectively $n$ and $m$. Suppose that $c=e_{1}\inact e_{2}\inact \ldots\inact e_{i}$ is an inactivation chain in $s$, where $e_{i}$ is a $p$-edge, for some player $p$. By Proposition \ref{prop-back2}, $c$ is an inactivation chain in the $p$-view $w$ of the initial segment $s'$ of $s$ terminating with the $p$-move corresponding to $e_{i}$. Since $s'\in\sigma$ when $p=P$, and $s'\in\tau$ when $p=O$,  by Definition \ref{def-backlevel}, $w$ has backtracking level less than or equal to $n$ if $p=P$ and less or equal than $m$ if $p=O$. We conclude that $i\leq n$ if $p=P$, and $i\leq m$ if $p=O$. 
\end{proof}

\section{The Trimming Operation}\label{sec-trimm}

In this Section we study the Trimming operation sketched in Section \S\ref{sec-proofidea}, we show that it transforms visible plays into visible plays and we determine the resulting backtracking level, thus removing the third obstacle in our way to the main results.

We will  show in Proposition \ref{prop-ord} that any visible play of real backtracking level $n$ can be interpreted as a sequence of visible plays of real backtracking level \emph{less} than $n$. These plays cannot grow infinitely and are compressed and stretched like springs by the player of backtracking level $n$, who, having a finite amount of energy, at some point must stop.
In order to identify such sequence of plays, we want to describe an operation of identifying in  any visible play $s$ of real backtracking level $n$ a subsequence which represents  a visible play $s'$ of real backtracking level strictly less than $n$. The idea is to erase whatever has been retracted by moves of backtracking level $n$: the sequence $s'$ is obtained from $s$ by removing all the moves that are in the interior of an edge from any move of backtracking level $n$ made by the principal player of $s$. More generally, we define an operation removing every move which is in the interior of an edge from a move of some fixed backtracking level made by a fixed player.

\begin{definition}[Trimming of a Play]\label{def-trim}
Suppose $s=s_{0}\, s_{1}\ldots\, s_{k}$ is any visible play, $n\in \mathbb{N}$, $n>0$ and $p$ a player. Let $\mu=(n,p)$. We define $\trim{\mu}{s}$ as the subsequence of $s$ obtained from $s$ by removing all the moves $s_{l}$ such that for some $i, j$ it holds: $a)$ $j < l < i$ and $\efrom{s_{i}}=(s_{j}, s_i)$; $b)$ $\blev(s_{i})=n$ and $\lambda(s_{i})=p$.
\end{definition}
It will be useful to exploit the following simple characterization of the Trimming operation.

\begin{lemma}\label{lemma-trim}
Suppose $s=s_{0}\, s_{1}\ldots\, s_{k}$ is any visible play, $p$ a player of backtracking level  $n$ in $s$  and $\mu=(n,p)$. 
\begin{enumerate}
\item
If  $\blev({s_{i}})\neq n$  or $\lambda(s_{i})\neq p$,  
 then 
$$ \trim{\mu}{(s_{0}\, s_{1}\ldots s_{i})}=\trim{\mu}{(s_{0}\, s_{1}\ldots s_{i-1})} s_{i}$$
\item If $\blev({s_{i}})=n$  and $\lambda(s_{i})=p$, assuming  $\efrom{s_{i}}=(s_{j}, s_{i})$, then
$$ \trim{\mu}{(s_{0}\, s_{1}\ldots s_{i})}=\comment{\trim{\mu}{(s_{0}\, s_{1}\ldots s_{j-1})}s_{j}\, s_{i}=}\trim{\mu}{(s_{0}\, s_{1}\ldots s_{j})}\, s_{i}$$

\item Assume 
$$s=r\,\circ_{z}\ldots\bullet_{z}  \ldots\circ_{2}\ldots\bullet_{2}\circ_{1}\ldots\bullet_{1}\, r'$$
and that: $r$ is the empty sequence or $r\in I$; $r'=s_{i} \ldots s_{k}$ and either $\bullet_{1}$ is in $\trim{\mu}{s}$ or  $s_{i}$ in $\trim{\mu}{s}$; for $1\leq i \leq z$, $\bullet_{i+1}$ immediately precedes $\circ_{i}$ and there is an edge $(\circ_{i}, \bullet_{i})$ in $s$. Then 
$$\trim{\mu}{s}=r\, \circ_{i_m}\ldots\bullet_{i_m}  \ldots\circ_{i_{2}}\ldots\bullet_{i_2}\circ_{i_1}\ldots\bullet_{i_{1}}\,r''$$
where $r''$ is a subsequence of $r'$ and $1\leq i_{1}< i_{2}\ldots < i_{m}\leq z$.

\end{enumerate}
\end{lemma}

\begin{proof}\mbox{}
\begin{enumerate}
\item 
By hypothesis $\blev(s_{i})\neq n$ or $\lambda(s_{i})\neq p$. Therefore the moves $s_{l}$, with $l< i$, that are in $\trim{\mu}{(s_{0}\,s_{1}\ldots s_{i})}$, are  exactly those in $\trim{\mu}{(s_{0}\,s_{1}\ldots s_{i-1})}$; since $s_{i}$ is in $\trim{\mu}{(s_{0}\,s_{1}\ldots s_{i})}$,  we obtain the thesis.
\vspace{1ex}
\item 
 For no $l, l'$, with $l' < j < l < i$, it is the case that $\blev(s_{l})=n$ and a $p$-edge $(s_{l'}, s_{l})$ is in $s$, otherwise $\efrom{s_{l}}\cross \efrom{s_{i}}$, which is impossible by Proposition \ref{prop-back3}. Therefore the moves $s_{l}$, with $l\leq j$, that are in $\trim{\mu}{(s_{0}\,s_{1}\ldots s_{i})}$, are  exactly those in $\trim{\mu}{(s_{0}\,s_{1}\ldots s_{j})}$. Moreover all the moves $s_{l}$, with $j<l<i$ are not in $\trim{\mu}{(s_{0}\,s_{1}\ldots s_{i})}$, which is the thesis.
\vspace{1ex}


 \item By hypothesis $$s=r\, \circ_{z}\ldots\bullet_{z}  \ldots\circ_{2}\ldots\bullet_{2}\circ_{1}\ldots\bullet_{1}\, r'$$ where, for $1\leq j \leq z$, $\bullet_{j+1}$ immediately precedes $\circ_{j}$ and there is an edge $(\circ_{j}, \bullet_{j})$ in $s$. We must show $$\trim{\mu}{s}=r\,\circ_{i_m}\ldots\bullet_{i_m}  \ldots\circ_{i_{2}}\ldots\bullet_{i_2}\circ_{i_1}\ldots\bullet_{i_{1}}\, r''$$
By hypothesis $r'=s_{i}\ldots s_{k}$ and either $\bullet_{1}$ is in $\trim{\mu}{s}$ or  $\circ_{0}:=s_{i}$ is in $\trim{\mu}{s}$; therefore,  every $p$ move $s_{l}$ of $s$, such that $\blev(s_{l})=n$ and $i<l$, points in $s$ to a move $s_{l'}$, with $i-1 \leq l'$. 
What we have to show now is that for every $l$, with $1\leq l\leq z$: i) if $\bullet_{l}$ is in $\trim{\mu}{s}$, then $\circ_{l}$ is in $\trim{\mu}{s}$ as well, and ii) if $\bullet_{l}$ is not in $\trim{\mu}{s}$, then  $\circ_{l}$ and all the moves between $\circ_{l}$ and $\bullet_{l}$ are not in $\trim{\mu}{s}$ either.
  $s$ is represented as:
$$\begin{tikzpicture}

\node at (3,0.04)   {$r$};
\node  (z) at (3.5,0) {$ \circ_{z}$};
\node  at (4,0) {$\ldots$};
\node (bk) at (4.5,0) {$\bullet_{z}$};
\node  at (5,0) {$\ldots$};
\node (ck-1) at (5.5,0) {$\circ_{l}$};
\node (z2) at (6.15,0) {$\ldots$};
\node (bk-1) at (6.7,0) {$\bullet_{l}$};
\node at (7.2, 0)   {$\ldots$};
\node (c2) at (7.7,0) {$\circ_{2}$};
\node (z2) at (8.2,0) {$\ldots$};
\node (b2) at (8.7,0) {$\bullet_{2}$};
\node  (c1) at (9,0) {$\circ_{1}$};
\node  at (9.5,0) {$\ldots$};
\node  (b1) at (10,0) {$\bullet_{1}$};
\node (d) at (10.5,0.09) {$r'$};
\draw[-latex,bend right]  (b1) edge (c1);
\draw[-latex,bend right]  (b2) edge (c2);
\draw[-latex,bend right]  (bk) edge (z);
\draw[-latex,bend right]  (bk-1) edge (ck-1);
\end{tikzpicture}$$

To show i), assume $\bullet_{l}$ is in $\trim{\mu}{s}$ and, for the sake of contradiction, that $\circ_{l}$ is not in it. Then there is a $p$ move $a$ in $s$  such that $\blev({a})=n$, $a$ points in $s$ to a move on the left of $\circ_{l}$, $a$ is on the right of $\circ_{l}$ but it is not the case that $a$ is on the right of $\bullet_{l}$. This means that $a$ lies between  $\circ_{l}$ and $\bullet_{l}$. But then $\efrom{a}\cross (\circ_{l}, \bullet_{l})$, impossible by Proposition \ref{prop-back3}.

To show ii), assume $\bullet_{l}$ is not in $\trim{\mu}{s}$ and, for the sake of contradiction, that either $\circ_{l}$ or some move between $\circ_{l}$ and $\bullet_{l}$ is in it. Then there is a $p$  move $a$ such that $\blev({a})=n$, and $(\circ_{l},\bullet_{l})\cross \efrom{a}$ or $a$ points to $\circ_{l}$. As before, $a$ cannot lie between $\circ_{h}$ and $\bullet_{h}$, for any $h< l$, otherwise $\efrom{a}\cross (\circ_{h}, \bullet_{h})$, impossible by Proposition \ref{prop-back3}. We have already showed that $a$ cannot lie on the right of the first move of $r'$. Therefore $a=\circ_{h}$, for some $0\leq h < l$. But by visibility of $s$, this is impossible: $\circ_{h}$ should point to one of the moves $\bullet_{1}, \bullet_{2},\ldots, \bullet_{z}, r$.
\end{enumerate}
\end{proof}
The Trimming operation takes a visible play $s$ and returns a sequence of moves $s^{\mu}$. It is not evident that $\trim{\mu}{s}$ is a visible play, for we have not yet studied in detail the shape of the views in $\trim{\mu}{s}$, 
nor it is immediate to assess what is the backtracking level of $\trim{\mu}{s}$. 
This study will be carried out in Proposition \ref{lem-trim} below, where we  establish that views in $\trim{\mu}{s}$ are nothing but subsequences of views in $s$.  Moreover,  if $\mu=(n,p)$ and $p$ has backtracking level $n$ in $s$, then $p$ has backtracking level less than or equal to  $n-1$ in $\trim{\mu}{s}$. Thus $\trim{\mu}{s}$ is exactly the subsequence of $s$ we were looking for: a visible play where $p$ plays only moves with backtracking level strictly less than $n$. 
\begin{proposition}[Properties of Trimming]\label{lem-trim}
Suppose $s=s_{0}\, s_{1}\ldots\, s_{k}$ is any visible play, $p$ a player of backtracking level $n>0$ in $s$ and $\mu=(n,p)$. Then:
\begin{enumerate}
\item For any player $q$, the $q$-view of any initial segment $IM$ of $\trim{\mu}{s}$ is a subsequence of the $q$-view of the initial segment $IS$ of $s$ that ends with the last move of $IM$. 
\item $\trim{\mu}{s}$ is a visible play.
\item $p$ is in $\trim{\mu}{s}$ of backtracking level less than or equal to $n-1$.
\end{enumerate}
\end{proposition}
\begin{proof}\mbox{}
\begin{enumerate}
\item By induction on $s$. Let $IM$ be any initial segment of $\trim{\mu}{s}$ and let  $s_{i}=\circ$ its last move. Let us fix also a player $q$. We have two cases:
\begin{itemize}
\item
$\lambda(\circ)=q$. We can write $s$ as
$$s=r \circ_{z}\ldots\bullet_{z}  \ldots\circ_{2}\ldots\bullet_{2}\circ_{1}\ldots\bullet_{1}\, \circ\, \ldots s_{k} $$
  where, for $1\leq i \leq z$, $\bullet_{i+1}$ immediately precedes $\circ_{i}$ and there is an edge $(\circ_{i}, \bullet_{i})$ in $s$ and $r\in I$ or $r$ is the empty sequence. By Lemma \ref{lemma-trim}, point 3,
$$\trim{\mu}{s}=r\circ_{i_m}\ldots\bullet_{i_m}  \ldots\circ_{i_{2}}\ldots\bullet_{i_2}\circ_{i_1}\ldots\bullet_{i_{1}}\circ\ldots s_{k}$$
for some sequence of numbers $1\leq i_{1}< i_{2}\ldots < i_{m}\leq z$. Therefore
$$IM=r\,\circ_{i_m}\ldots\bullet_{i_m}  \ldots\circ_{i_{2}}\ldots\bullet_{i_2}\circ_{i_1}\ldots\bullet_{i_{1}}\circ$$
$$IS=r\, \circ_{z}\ldots\bullet_{z}  \ldots\circ_{2}\ldots\bullet_{2}\circ_{1}\ldots\bullet_{1}\, \circ$$
Now, assuming that $r'$ is the $q$-view of $r$, the $q$-view of $IM$ is 
$$r'\,\circ_{i_{m}}\,\bullet_{i_{m}}  \ldots\circ_{i_2}\,\bullet_{i_2}\circ_{i_1}\,\bullet_{i_1}\, \circ$$
while the $q$-view of $IS$ is 
$$r'\,\circ_{z}\,\bullet_{z}  \ldots\circ_{2}\,\bullet_{2}\circ_{1}\,\bullet_{1}\, \circ$$ 
 which is what we wanted to show.
  
\item  $\lambda(\circ)\neq q$. Letting $\circ_{1}=\circ$, we can write $s$ as
$$s=r \bullet_{z}\ldots\circ_{z}  \ldots\bullet_{2}\ldots\circ_{2}\bullet_{1}\ldots\circ_{1}\ldots s_{k} $$
  where, for $1\leq i \leq z$, $\circ_{i+1}$ immediately precedes $\bullet_{i}$ and there is an edge $(\bullet_{i}, \circ_{i})$ in $s$ and $r\in I$ or $r$ is the empty sequence. As $\circ_{1}$ is in $\trim{\mu}{s}$, by Lemma \ref{lemma-trim}, point 3, we have
$$\trim{\mu}{s}=r\bullet_{i_m}\ldots\circ_{i_m}  \ldots\bullet_{i_{1}}\ldots\circ_{i_1}\bullet_{1}\ldots\circ_{{1}}\ldots s_{k}$$
for some sequence of numbers $1< i_{1}< i_{2}\ldots < i_{m}\leq z$. Therefore
$$IM=r\,\bullet_{i_m}\ldots\circ_{i_m}  \ldots\bullet_{i_{1}}\ldots\circ_{i_1}\bullet_{1}\ldots\circ_{{1}}$$ 
$$IS=r\, \bullet_{z}\ldots\circ_{z}  \ldots\bullet_{2}\ldots\circ_{2}\bullet_{1}\ldots\circ_{1}$$
Now,  assuming that $r'$ is the $q$-view of $r$, the $q$-view of $IM$ is 
$$r'\,\bullet_{i_{m}}\,\circ_{i_{m}}  \ldots\bullet_{i_1}\,\circ_{i_1}\bullet_{1}\,\circ_{1}$$
while the $q$-view of $IS$ is 
$$r'\,\bullet_{z}\,\circ_{z}  \ldots\bullet_{2}\,\circ_{2}\bullet_{1}\,\circ_{1}$$ 
 which is what we wanted to show.
\end{itemize}

\vspace{1ex}
\item
First of all, $\trim{\mu}{s}$ is a justified sequence: if in $s$ there is an edge $(s_{j}, s_{i})$ and $s_{i}$ is in $\trim{\mu}{s}$, then $s_{j}$ must be in $\trim{\mu}{s}$ as well, otherwise there would be an index $l$ such that $j<l<i$, $\lambda(s_{l})=p$, $\blev(s_{l})=n$ in $s$ and $\efrom{s_{l}}\cross (s_{j}, s_{i})$; by Proposition \ref{prop-back3}, that is impossible.\\
To show that $\trim{\mu}{s}$ is visible, assume that $\circ$ is the last move of an initial segment $IM$ of $\trim{\mu}{s}$, with $\lambda(\circ)=q$, and that $IS$ is the initial segment of $s$ ending with the last move of $IM$. We can write $s$ as
$$s=r \circ_{z}\ldots\bullet_{z}  \ldots\circ_{2}\ldots\bullet_{2}\circ_{1}\ldots\bullet_{1}\, \circ\, \ldots s_{k} $$
  where, for $1\leq i \leq z$, $\bullet_{i+1}$ immediately precedes $\circ_{i}$ and there is an edge $(\circ_{i}, \bullet_{i})$ in $s$ and $r\in I$ or $r$ is the empty sequence.
  By Lemma \ref{lemma-trim}, point 3.,
$$\trim{\mu}{s}=r\circ_{i_m}\ldots\bullet_{i_m}  \ldots\circ_{i_{2}}\ldots\bullet_{i_2}\circ_{i_1}\ldots\bullet_{i_{1}}\circ\ldots s_{k}$$
for some sequence of numbers $1\leq i_{1}< i_{2}\ldots < i_{m}\leq z$. Therefore
$$IM=r\,\circ_{i_m}\ldots\bullet_{i_m}  \ldots\circ_{i_{2}}\ldots\bullet_{i_2}\circ_{i_1}\ldots\bullet_{i_{1}}\circ$$
$$IS=r\, \circ_{z}\ldots\bullet_{z}  \ldots\circ_{2}\ldots\bullet_{2}\circ_{1}\ldots\bullet_{1}\, \circ$$
Now,  assuming that $r'$ is the $q$-view of $r$, by point 1., the $q$-view of $IM$ is 
$$r'\,\circ_{i_{m}}\,\bullet_{i_{m}}  \ldots\circ_{i_2}\,\bullet_{i_2}\circ_{i_1}\,\bullet_{i_1}\, \circ$$
while the $q$-view of $IS$ is 
$$r'\,\circ_{z}\,\bullet_{z}  \ldots\circ_{2}\,\bullet_{2}\circ_{1}\,\bullet_{1}\, \circ$$ 
We have to show that $\circ$ points in $IM$ to $r$ or some $\bullet_{i_{n}}$, with $1\leq n\leq m$. Since $s$ is visible, $\circ$ points in $IS$ to $r$ or some $\bullet_{l}$, with $1\leq l\leq z$. But since $IM$ is justified, $\bullet_{l}$ is in $IM$, and thus $l=i_{n}$,  with $1\leq n\leq z$. \\

\item We first show that for every edge $e=(s_{j}, s_{i})$ of $\trim{\mu}{s}$, 
$e$  is active in an initial segment $r'$ of $\trim{\mu}{s}$ if and only if $e$ is active in the initial segment $r$ of $s$ ending with the last move of $r'$. We proceed by induction on $k-i$.

 If  $i=k$, then $e=(s_{j}, s_{k})$, $r'=\trim{\mu}{s}$ and $r=s$, so $e$ must be active both in $r'$ and $r$, for it is the last edge of both plays. 
 
 If  $i<k$, we prove separately the two directions of the implication.
 \begin{itemize}
 \item
 For one direction, suppose that $e$ is active in $r'$: we want to show that $e$ is active also in $r$, and thus that if $e'=(s_{j'}, s_{i'})$ is an edge of $r$ such that $e\cross e'$, then $e'$ is inactive in $r$. Now, if $e'$ is in $r'$, then $e'$ is inactive in $r'$, and by induction hypothesis $e'$ is inactive also in $r$. So suppose $e'$ is not in $r'$. Then, there is a $p$-edge $e''=(s_{j''}, s_{i''})$ in $s$ such that $\blev(s_{i''})=n$ in $s$ and $j''<i'<i''$. A picture may be useful to summarize how  $s$ looks like:
 
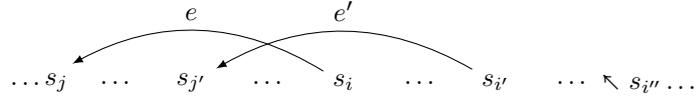
\vspace{-1ex}
 $$\begin{tikzpicture}
\node (a) at (0,0) {$\ldots s_{j}$};
\node (z1) at (1,0) {$\ldots$};
\node (b) at (2,0) {$s_{j'}$};
\node at (2, 0.9)  {$e$};
\node (z2) at (3,0) {$\ldots$};
\node (c) at (4,0) {$s_{i}$};
\node at (4, 0.95)  {$e'$};
\node (z3) at (5,0) {$\ldots$};
\node (d) at (6,0) {$s_{i'}$};
\node (z4) at (7,0) {$\ldots$};
\node (e) at (8,0) {$\efrom{s_{i''}}\ldots$};
\draw[-latex,bend right]  (d) edge (b);
\draw[-latex,bend right]  (c) edge (a);
\end{tikzpicture}$$
 We are first going  to determine where $s_{i''}$ points to, that is, where $s_{j''}$ lies. Since $s_{i}$ is in $r'$, it holds $i\leq j''$ and we already know that $j''<i'$; therefore    $e'\cross e''$.
 Hence $s$ is of the shape
 
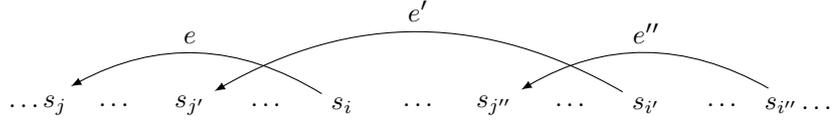
\vspace{-2ex}
  $$\begin{tikzpicture}
\node (a) at (0,0) {$\ldots s_{j}$};
\node (z1) at (1,0) {$\ldots$};
\node (b) at (2,0) {$s_{j'}$};
\node (z2) at (3,0) {$\ldots$};
\node (c) at (4,0) {$s_{i}$};
\node at (2, 0.9)  {$e$};
\node (z3) at (5,0) {$\ldots$};
\node at (5, 1.25)  {$e'$};
\node (f) at (6,0) {$ s_{j''}$};
\node (z4) at (7,0) {$\ldots$};
\node (d) at (8,0) {$s_{i'}$};
\node at (8, 0.95)  {$e''$};
\node (z5) at (9,0) {$\ldots$};
\node (e) at (10,0) {$s_{i''}\ldots$};
\draw[-latex,bend right]  (d) edge (b);
\draw[-latex,bend right]  (c) edge (a);
\draw[-latex,bend right]  (e) edge (f);
\end{tikzpicture}$$

  Finally, let $s_{l}$ be the last move of $r$. By hypothesis, $s_{l}$ is also in $r'$. Moreover, since $i'\leq l$, it is not the case that $l < i''$: otherwise, $s_{l}$ would not be in $r'$.  We conclude that  $e''$ is in $r$ and, since by Proposition \ref{prop-back3} $e''$ is active in $r$, by definition $e'$ is inactive in $r$. 
 
 \item For the other direction, suppose that $e$ is active in ${r}$: we want to show that $e$ is active also in $r'$, and thus that if $e'$ is an edge of $r'$ such that $e\cross e'$, $e'$ is inactive in $r'$. Indeed, since $e'$ is in $r'$, it is as well in $r$; hence, $e'$ is inactive in $r$. By induction hypothesis, $e'$ is inactive in $r'$ as well. 

\end{itemize}

Suppose now that $e\inact e'$ in $\trim{\mu}{s}$; we want to show that $e\inact e'$ in $s$ as well. Indeed, let us consider the shortest initial segment $r'$ of $\trim{\mu}{s}$ that contains $e$ and $e'$; let $r$ be the initial segment of $s$ ending with the last move of $r'$. By Proposition \ref{prop-back1}, $e=(\bullet, \circ)$ and $e'=(\bullet', \circ')$, with $\lambda(\circ)=\lambda(\circ')$. Then for some move $\bullet''$, we have
\vspace{-1ex}
$$r=\ldots \bullet\ldots\bullet'\ldots \circ\ldots \bullet''\circ'$$
and we claim that $r'$ must be of the shape
\vspace{-1ex}
$$r'=\cdots \bullet\cdots\bullet'\cdots \circ\cdots \bullet''\circ'$$
possibly with fewer moves than $r$: we have just to check that $\bullet''$ must be in $r'$. Indeed, we have $\lambda(\circ')\neq p$ or $\blev(\circ')\neq n$ in $s$, otherwise $\circ$ would not be in $r'$; moreover, no move in $s$ different from $\circ'$ can point to the left of $\bullet''$ without pointing to the left of $\circ'$, which implies that if $\bullet''$ was not in $r'$, also $\circ'$ would not in $r'$, a contradiction. 
Since $e\inact e'$ in $\trim{\mu}{s}$, $e$ is active in the sequence obtained from $r'$ by removing the last move: $\cdots \bullet\cdots\bullet'\cdots \circ\cdots \bullet''$; therefore, by what we have just proved, $e$ is active in the sequence obtained from $r$ by removing the last move: $\ldots \bullet\ldots\bullet'\ldots \circ\ldots \bullet''$. We conclude that $e\inact e'$ in $r$ and thus in $s$.

We are finally able to prove the thesis: suppose $e$ is a $p$-edge from a move of backtracking level $n-1$ in $\trim{\mu}{s}$; we must show that for no edge $e'$, $e\inact e'$ in $\trim{\mu}{s}$. Indeed, suppose $e\inact e'$ in $\trim{\mu}{s}$. Then, there is in $\trim{\mu}{s}$ an inactivation chain $c=e_{1}\inact \ldots \inact e_{n-2}\inact e\inact e'$. We have proved that $c$ must be an inactivation chain also in $s$. Therefore, by Proposition \ref{prop-back1}, $e'$ is a $p$-edge from a move of backtracking level $n$ in $s$; since $e\cross e'$, $e$ should not be in $\trim{\mu}{s}$, contradicting our hypothesis.
\end{enumerate}
\end{proof}

\vspace{-2ex}
\section{Upper Bounds on the Length of Interactions}\label{sec-bounds}
In this Section, we establish our main results about the length of visible plays and interactions between bounded strategies. 

We first introduce another numerical measure $\com{s}{i}$ of how complex is a move $s_{i}$ in a play $s$, which is very different from the backtracking level of $s_{i}$: it represents how advanced the move is in the current view. Given this association between moves of a play and numbers, we can map each visible play $s$ into a sequence of numbers, which represents its complexity. The complexity $\comu{\mu}{s}$ of $s^{\mu}$ is instead defined as the restriction of the complexity of $s$ to $s^{\mu}$: in other terms, $\comu{\mu}{s}$ is the sequence of the complexities that the moves of $\trim{\mu}{s}$ have \emph{in} $s$. The lexicographic ordering $\prec$ of finite sequences of numbers is then used to order the complexities of plays.

\begin{definition}[Complexity of Moves and Plays, Lexicographic Ordering $\prec$]\label{def-ord}Let $s=s_{0}\, s_{1} \ldots s_{k}$ be a visible play.
\begin{itemize}

\item Assume $0\leq i\leq k$. We define the complexity of $s_{i}$ in $s$, denoted with $\com{s}{i}$, as the length of $\pview{s_{0}\, s_{1}\ldots s_{i}}$ if $\lambda(s_{i})=P$, as the length of $\oview{s_{0}\, s_{1}\ldots s_{i}}$ if $\lambda(s_{i})=O$.
\item Let $n\in\mathbb{N}$,  $p$ be a player and $\mu=(n, p)$. Suppose $\trim{\mu}{s}=s_{i_{0}}\, s_{i_{1}}\ldots s_{i_{m}}$. We define 
$$\comu{\mu}{s} = \com{s}{i_{0}}\, \com{s}{i_{1}}\ldots \com{s}{i_{m}}$$
\item
Let $s=a_{0}\, a_{1} \ldots a_{n}$ and $r=b_{0}\, b_{1} \ldots b_{m}$ be two sequences of natural numbers.  We define that 
$s \prec r$  if and only if either
$$n<m\ \land\ \forall i.\ 0\leq i\leq n \implies a_{i}=b_{i}$$ 
or 
$$\exists k.\ 0\leq k\leq n\, \land \, (\forall i.\ 0\leq i < k\,\implies a_{i}=b_{i})\ \land\  a_{k}< b_{k}$$
 \end{itemize}
\end{definition}
The next Proposition \ref{prop-ord0} is simple, but is really crucial: it says that if a player answers an adversary move, this answer is more complex than the player's first reaction to that move and than any of his other previous non-retracted answers to that move.  Intuitively, changing one's mind  and backtracking costs, and this cost prevents the backtracking  during a play between bounded strategies to be unlimited. This phenomenon was also noticed and exploited by Coquand (\cite{Coquand}, pp. 332).
\begin{proposition}\label{prop-ord0}
Let  $s=s_{0}\, s_{1}\ldots s_{k}$ be a visible play and $(s_{j}, s_{i})$ is a $p$-edge of $s$, with $i\neq j+1$. Then: \begin{enumerate}
 \item $\com{s}{j+1}< \com{s}{i}$.  
 \item Suppose that $j<l<i$, that $(s_{j}, s_{l})$ is an edge of $s$ and for no egde $e$ of $s$, $(s_{j}, s_{l})\cross e$. Then $\com{s}{l}< \com{s}{i}$.
\end{enumerate}
\end{proposition}
\begin{proof}
By Proposition \ref{prop-view}
$$s_{0}\ldots s_{i}\ =\ s_{0}\ldots s_{j} \circ_{z}\ldots\bullet_{z}  \ldots\circ_{2}\ldots\bullet_{2}\circ_{1}\ldots\bullet_{1}\, s_{i} $$
  where, for $1\leq n \leq z$, $\bullet_{n+1}$ immediately precedes $\circ_{n}$ and there is an edge $(\circ_{n}, \bullet_{n})$ in $s$.  Since $(s_{j}, s_{l})$ is crossed by no edge of $s$, it must be the case that $s_{l}=\circ_{h}$, with $1\leq h \leq z$. Thus to prove 1. and 2. simultaneously, it is enough to assume $s_{n}=\circ_{h}$, with $1\leq h \leq z$, and show $\com{s}{n}< \com{s}{i}$.
Indeed, let $v$ be the $p$-view of $s_{0}\ldots s_{j}$; $s_{i}$ and $\circ_{h}$ are $p$ moves,  so the $p$-view of $s_{0}\ldots s_{i}$ is $$v\, \circ_{z}\,\bullet_{z}  \ldots \circ_{h+1}\,\bullet_{h+1}\circ_{h}\,\bullet_{h}\ldots \circ_{2}\,\bullet_{2}\circ_{1}\,\bullet_{1}\, s_{i}$$ 
and the $p$-view of $s_{0}\ldots \circ_{h}$ is 
$$v\,  \circ_{z}\,\bullet_{z}  \ldots \circ_{h+1}\,\bullet_{h+1}\circ_{h}$$ As $s_{n}=\circ_{h}$,
we get  $\com{s}{n}< \com{s}{i}$. 
\end{proof}
We now prove that any visible play $s_{0}\, s_{1}\ldots s_{k}$ in which some player $p$ has backtracking level $n$ can be seen as a sequence of visible plays in which $p$ has backtracking level strictly less than $n$. The plays of the sequence are obtained just by applying the Trimming operation, with $\mu=(n,p)$, from left to right to all initial segments of $s$. The plays are each less complex than the next one, according to the relation $\prec$. 
More precisely, define $r_{i}= s_{0}\, s_{1}\ldots s_{i}$. If we imagine to watch the sequence $\trim{\mu}{(r_{0})}, \trim{\mu}{(r_{1})}, \ldots, \trim{\mu}{(r_{k})}$ as a succession of photograms, we would see that for some time the $\trim{\mu}{(r_{i})}$ are extended by one element, then they are reduced by removing some final segment, then they are extended again, and so forth. Something like a spring, which is stretched for some time, then suddenly is compressed only to be stretched again and so on: a sort of $1$-backtracking game \cite{BerCoqHay}, as we observed in Section \S\ref{sec-proofidea}. Moreover, this series of extensions and compressions come with a cost: $\comu{\mu}{r_{0}}\prec \comu{\mu}{r_{1}}\prec \ldots \prec \comu{\mu}{r_{k}}$. 
\begin{proposition}[Spring Property]\label{prop-ord}
Let $s=s_{0}\, s_{1}\ldots s_{k}$ be a visible play, $p$ a player of backtracking level $n>0$ in $s$ and $\mu=(n,p)$. Then, whenever $1\leq i< k$, it holds
$$\comu{\mu}{s_{0}\, s_{1}\ldots s_{i}}\prec \comu{\mu}{s_{0}\, s_{1}\ldots s_{i+1}}$$
\end{proposition}

\begin{proof}
There are two possibilities: either $\blev({s_{i+1}})=n$   and $\lambda(s_{i+1})=p$ or not. 

i) Suppose $\blev({s_{i+1}})=n$  and $\lambda(s_{i+1})=p$. Assume $\efrom{s_{i+1}}=(s_{j}, s_{i+1})$. Then $j\neq i$, since $n>0$, and 
by Lemma \ref{lemma-trim} (point 2. and 1.)
$$ \trim{\mu}{(s_{0}\, s_{1}\ldots s_{i+1})}=\trim{\mu}{(s_{0}\, s_{1}\ldots s_{j-1})}s_{j}\, s_{i+1}$$
Moreover, $\lambda(s_{i})=\lambda(s_{j})\neq p$ and for every $l$ such that $j<l\leq i$, $\lambda(s_{l})=p$ and $\blev(s_{l})= n$ in $s$, it is not the case that $\efrom{s_{l}}=(s_{h}, s_{l})$ with $h<j$, otherwise $\efrom{s_{l}}\cross \efrom{s_{i+1}}$ and by Proposition \ref{prop-back3} it is impossible. Therefore, since $i\neq j+1$, there is an $l$ such that $j<l<i$
 $$ \trim{\mu}{(s_{0}\, s_{1}\ldots s_{i})}=\trim{\mu}{(s_{0}\, s_{1}\ldots s_{j-1})}s_{j}\,s_{l}\ldots s_{i}$$
Now, either $l=j+1$ or $\lambda(s_{l})=p$, $\blev(s_{l})=n$ in $s$ and $(s_{j}, s_{l})$ is in $s$
 (because if we assume $l\neq j+1$, then $s_{j+1}$ is not in $ \trim{\mu}{(s_{0}\, s_{1}\ldots s_{i})}$, thus it must hold that for some $h$, with $j+1<h\leq l$, $\lambda(s_{h})=p$, $\blev(s_{h})=n$ and $(s_{j}, s_{h})$ is in $s$; moreover, $h=l$,  because $s_{h}$ must be  in $ \trim{\mu}{(s_{0}\, s_{1}\ldots s_{i})}$, since $\efrom{s_{h}}$ is crossed by no edge by Proposition \ref{prop-back3}); 
 in either case, by Propositions  \ref{prop-back3} and \ref{prop-ord0}, $\com{s}{l}<\com{s}{i+1}$, thus we get 
$$
\begin{aligned}
\comu{\mu}{s_{0}\, s_{1}\ldots s_{i}}=&\ \comu{\mu}{s_{0}\, s_{1}\ldots s_{j-1}}\, \com{s}{{j}}\, \com{s}{{l}}\ldots \com{s}{{i}}\\
\prec &\ \comu{\mu}{s_{0}\, s_{1}\ldots s_{j-1}}\, \com{s}{{j}}\, \com{s}{{i+1}}\\
= &\ \comu{\mu}{s_{0}\, s_{1}\ldots s_{i+1}}
\end{aligned}
$$


ii) Suppose $\blev({s_{i+1}})\neq n$  or $\lambda(s_{i+1})\neq {p}$. Then, by Lemma \ref{lemma-trim} (point 1.)
$$ \trim{\mu}{(s_{0}\, s_{1}\ldots s_{i+1})}=\trim{\mu}{(s_{0}\, s_{1}\ldots s_{i})} s_{i+1}$$
 Therefore
 $$\comu{\mu}{s_{0}\, s_{1}\ldots s_{i}}\prec \comu{\mu}{s_{0}\, s_{1}\ldots s_{i}}\, \com{s}{i+1} = \comu{\mu}{s_{0}\, s_{1}\ldots s_{i+1}}$$

\end{proof}
Finally, we are able to prove our first main Theorem, which tells that the length of any visible play of real backtracking level $b$ is bounded by a tower of exponentials of height $b+1$. More precisely, define the \textbf{hyperexponential function} $a_{n}^{m}$, with $a,n,m\in \mathbb{N}$, by: $a_{0}^{m}=m$ and $a_{n+1}^{m}=a^{a_{n}^{m}}$; then we prove the following statement.
\begin{theorem}[Abstract Min-Backtracking-Theorem]\label{thm-abstB}
Suppose $s=s_{0}\, s_{1}\ldots s_{l}$ is a visible play of real backtracking level $b$, $p$ is the principal player of $s$, $k$ is the length of the longest $p$-view or $\overline{p}$-view  of an initial segment of $s$.  Then
$$l\ \leq \ \underbrace{k^{k^{.^{.^{k^{k}}}}}}_{b+1} \leq \ \underbrace{2^{2^{.^{.^{2^{k(\log k)\cdot 2}}}}}}_{b+1} =2_{b}^{k(\log k)2} $$ 
\end{theorem}
\begin{proof} We prove the first inequality, by induction on $b$. 
\begin{itemize}
\item Case $b=0$. For every $p$ move $s_{j}$, with $j>0$,  $\blev(s_{j})=0$, so one has by definition $\efrom{s_{j}}=(s_{j-1}, s_{j})$. Therefore, for every $i$, the $\overline{p}$-view of $s_{0}\,s_{1}\ldots s_{i}$ is  $s_{0}\,s_{1}\ldots s_{i}$. Since the length of $\overline{p}$-views is bounded by $k$, we conclude $l=k$, which is what we wanted to show. \\

\item Case $b>0$. Let $\mu=(b,p)$. For every $i$, with $0\leq i\leq l$, let $r_{i}=s_{0}\,s_{1}\ldots s_{i}$. By Proposition \ref{prop-ord}
$$\comu{\mu}{r_{0}}\prec \comu{\mu}{r_{1}}\prec \comu{\mu}{r_{2}}\prec \ldots \prec \comu{\mu}{r_{l}}$$
By Proposition \ref{lem-trim}, for every $i$, $p$ has in $\trim{\mu}{(r_{i})}$ backtracking level  less than or equal to $b-1$ and no view of no initial segment of $\trim{\mu}{(r_{i})}$ is longer than $k$, thus by induction hypothesis  $\trim{\mu}{(r_{i})}$ is no longer than 
 $$ n=\underbrace{k^{k^{.^{.^{k^{k}}}}}}_{b}$$
Moreover,  for every $0\leq m\leq l$, $\com{s}{m}\leq k$. Therefore the set of sequences
$$\{\comu{\mu}{r_{i}}\ |\  0\leq i\leq l\}$$
has cardinality less than or equal to $k^{n}$. By Definition \ref{def-ord} of $\prec$, we conclude that $l\leq k^{n}$, which is the first inequality we wanted to show.\\
\end{itemize}
Now, let us prove by induction on $i\in\mathbb{N}$ that $k_i^{k}\leq 2_i^{(k+i)\log k}$. If $i=0$, then $k_i^{k}=k$ and $2_i^{(k+i)\log k}=k\log k$, and we are done. Moreover, by induction hypothesis and standard inequalities:
$$\begin{aligned}&k_{i+1}^{k}= k^{k_{i}^{k}}= 2^{k_{i}^{k}\log k}\leq\\
\leq &\, 2^{2_{i}^{(k+i)\log k}\log k}\leq\, 2^{2_{i}^{(k+i)\log k}\,2_i^{\log k}}\leq\, 2^{2_{i}^{(k+i)\log k+\log k}}=2^{2_{i}^{(k+i+1)\log k}}=2_{i+1}^{(k+i+1)\log k}
\end{aligned}
$$
Since by Proposition \ref{prop-back2}, $b\leq k$, we conclude 
$$l\leq\, k_b^k\leq\, 2_b^{(k+b)\log k}\leq\, 2_b^{(k+k)\log k} $$
\end{proof}
We now readily prove our second main Theorem:  bounded strategies have only finite interactions, whose length is bounded by a tower of exponentials of height equal to the minimum among the backtracking levels of the strategies plus one.  
\begin{theorem}[Min-Backtracking-Theorem]\label{thm-B}
Let $\sigma$ be a Player $n$-backtracking strategy bounded by $k$, $\tau$  an Opponent $m$-backtracking strategy bounded by $k$ and $b=\mathsf{min}(n,m)$. Assume $s\in\sigma\star\tau$ and $s=s_{0}\, s_{1}\ldots s_{l}$. Then 
$$l\ \leq \ \underbrace{k^{k^{.^{.^{k^{k}}}}}}_{b+1} \leq \underbrace{2^{2^{.^{.^{2^{k(\log k)\cdot 2}}}}}}_{b+1}=2_{b}^{k(\log k)2}$$
\end{theorem}
\begin{proof}
By Theorem \ref{thm-maxmin}, $s$ is visible and of real backtracking level $b$. By Definition \ref{def-stratvisint} of $\sigma\star\tau$, every initial segment $s'$ of $s$ is in $\sigma$ or $\tau$ and since both strategies are bounded by $k$, $\pview{s'}\leq k$ and $\oview{s'}\leq k$. By Theorem \ref{thm-abstB}, 
$$l\ \leq \ \underbrace{k^{k^{.^{.^{k^{k}}}}}}_{b+1}\leq \underbrace{2^{2^{.^{.^{2^{k(\log k)\cdot 2}}}}}}_{b+1}$$
\end{proof}

\section{Backtracking Level and Arena Depth}\label{sec-gadget}
In this Section, we study the relationship between backtracking level and arena depth, removing the fourth and last obstacle in our way to our main result. One would expect that all plays and strategies relative to an arena of depth $d$ have backtracking level at most equal to $d$. But the following example shows that a badly specified strategy may have backtracking level $6$ in an arena of depth $5$:
 $$
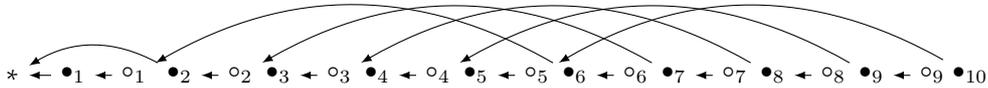
\begin{tikzpicture}
 \node (0) at (-0.9, 0) {$*$};
 \node (x) at (-0.1,0) {$\bullet_{1}$};
\node (a) at (0.7,0) {$\circ_{1}$};
\node (b) at (1.3,0) {$\bullet_{2}$};
\node (c) at (2.1,0) {$\circ_{2}$};
\node (d) at (2.6,0) {$\bullet_{3}$};
\node (e) at (3.4,0) {$\circ_{3}$};
\node (f)  at (3.9, 0)  {$\bullet_{4}$};
\node (g)  at (4.7, 0)  {$\circ_{4}$};
\node (h) at (5.2,0) {$\bullet_{5}$};
\node (i) at (6,0) {$\circ_{5}$};
\node (l) at (6.5,0) {$\bullet_{6}$};
\node (m) at (7.3,0) {$\circ_{6}$};
\node (n)  at (7.8, 0)  {$\bullet_{7}$};
\node (o)  at (8.6, 0)  {$\circ_{7}$};
\node (p) at (9.1,0) {$\bullet_{8}$};
\node (q) at (9.9,0) {$\circ_{8}$};
\node (r) at (10.4,0) {$\bullet_{9}$};
\node (s) at (11.2,0) {$\circ_{9}$};
\node (t)  at (11.7, 0)  {$\bullet_{10}$};
\draw[-latex, right]  (x) edge (0);
\draw[-latex, right]  (a) edge (x);
\draw[-latex, right]  (c) edge (b);
\draw[-latex, right]  (e) edge (d);
\draw[-latex, right]  (g) edge (f);
\draw[-latex, right]  (i) edge (h);
\draw[-latex, right]  (m) edge (l);
\draw[-latex, right]  (o) edge (n);
\draw[-latex, right]  (q) edge (p);
\draw[-latex, right]  (s) edge (r);
\draw[-latex,bend right]  (b) edge (0);
\draw[-latex,bend right]  (l) edge (a);
\draw[-latex,bend right]  (n) edge (c);
\draw[-latex,bend right]  (p) edge (e);
\draw[-latex,bend right]  (r) edge (g);
\draw[-latex,bend right]  (t) edge (i);
\end{tikzpicture}$$
For the sake of legibility we have not displayed the edges from the moves $\bullet_3,\bullet_{4}, \bullet_{5}$, which are all supposed to point to $*$. We have $\efrom{\bullet_{2}}\inact\efrom{\bullet_{6}}\inact \efrom{\bullet_{7}}\inact \efrom{\bullet_{8}}\inact \efrom{\bullet_{9}}\inact \efrom{\bullet_{10}}$, therefore $\bullet_{10}$ is of backtracking level $6$.  Of course, there is nothing special about the number 6: this construction can be generalized to produce strategies with as great a backtracking level as one wishes. The pattern is: the first $n$ black moves point to $*$, the following $n$ black moves point to the move that comes exactly $2(n-1)$ moves before.

This example notwithstanding, we are going to show that, by inserting some dummy moves into any play relative to an arena of depth $d\geq 2$, we can lower the backtracking level of the play to $d-2$. \\
How? For energy saving reasons, turning off the light when leaving a room is in order. Likewise, to avoid that the backtracking level of a strategy or of a play grows arbitrarily, it is enough to inactivate backtracking edges immediately after they have been played. This idea can be implemented by inserting in any play some administrative moves, whose only purpose is to lower the backtracking level of a selected player, in our case $O$. We insert a pair of moves $\off\;\on$ immediately after every $P$ move. In this way, before any non-starting $O$  move $\bullet$ there will be the move $\on$ and immediately after the $P$ move $\circ$ that follows $\bullet$ there will be the move $\off$. The sequence $\on\,\bullet\,\circ\,\off$ forms what we call an \textbf{inactivation gadget}: $\on$ ``turns on'' the edge from $\bullet$, while $\off$  ``turns it off''. We would want the move $\off$ to simply point back to the nearest occurrence of $\on$. That works when the play is an $O$-view, but in general, to preserve visibility, $\off$ must refer back to an occurrence of $\on$ which is placed much before.

\begin{definition}[Inactivation Gadget]\label{def-gadget} 
We add to our fixed arena $\mathbb{A}$ special moves $\on$ and $\off$ which, for simplicity,  are justified by any other move of opposite polarity and $\lambda(\on)=P$ and $\lambda(\off)=O$. We denote with $\mathbb{A}^{+}$ the enlarged arena.
Given a play $s=\bullet_{0}\,\circ_{0}\,\bullet_{1}\,\circ_{1}\ldots \bullet_{k}\,\circ_{k}\, \bullet_{k+1}$, we define  
$$s^{\medtriangledown}=  \bullet_{{0}}\, \circ_{{0}}\, \off\; \on\, \bullet_{{1}}\, \circ_{{1}}\, \off\;\on  \ldots \bullet_{{k}}\, \circ_{k}\, \off\; \on\, \bullet_{k+1}$$
which is always assumed to be a play such that: the moves $\bullet_{i}$ and $\circ_{i}$ for $0\leq i\leq k+1$ point to the same moves they point to in $s$; the first occurrence of $\off$ points to $\circ_{0}$; for every edge $(\bullet_{i}, \circ_{j})$ in $s$, with $j>0$, the occurrence of the move $\off$ immediately following $\circ_{j}$,  points to the occurrence of the move $\on$ immediately preceding $\bullet_{i}$;   each occurrence of the move $\on$ points to the immediately preceding move $\off$. In a picture, when $j>0$
 $$
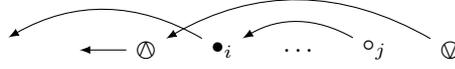
\begin{tikzpicture}
\node (a) at (1,0) {$$};
\node (b) at (2,0) {$\on$};

\node (c) at (3,0) {$\bullet_{i}$};
\node       at (4,0) {$\ldots$};
\node (d) at (5,0) {$\circ_{j}$};

\node (e) at (6,0) {$\off$};
\draw[-latex, right]  (b) edge (a);
\draw[-latex,bend right]  (d) edge (c);
\draw[-latex,bend right]  (e) edge (b);
\draw[-latex,bend right]  (c) edge (x);
\end{tikzpicture}$$
We shall denote with $\on^{(i)}$ and $\off^{(i)}$ respectively the $i$-th occurrence of $\on$ and  the $i$-th occurrence of $\off$ in $s^{\medtriangledown}$.
\end{definition}

We now characterize the views of $s^{\medtriangledown}$ and prove that $s^{\medtriangledown}$ is a visible play over $\mathbb{A}^{+}$ whenever $s$ is a visible play over $\mathbb{A}$. In particular we show that the operation of $O$-view commutes with $^{\medtriangledown}$, which means: $\oview{s^{\medtriangledown}}=\oview{s}^{\medtriangledown}$.  Moreover, the operation of $P$-view eliminates all the gadgets:   $\pview{s^{\medtriangledown}}=\pview{s}$.
\begin{proposition}[Views and Visibility of $s^{\medtriangledown}$]\label{prop-gadgetvis}
Let $s=\bullet_{0}\,\circ_{0}\,\bullet_{1}\,\circ_{1}\ldots \bullet_{k}\,\circ_{k}\, \bullet_{k+1}$ be a visible play over $\mathbb{A}$.
\begin{enumerate}
\item Assume $0\leq j\leq k$ and there are indexes $i_{1},j_{1},\ldots, i_{m}, j_{m}$ such that
$$\oview{\bullet_{0}\,\circ_{0}\,\bullet_{1}\,\circ_{1}\ldots \bullet_{j}\,\circ_{j}}=\bullet_{i_{1}}\, \circ_{j_{1}}\ldots\bullet_{{i_{m}}}\, \circ_{j_{m}}$$
Then
$$\oview{(\bullet_{{0}}\, \circ_{{0}}\, \off\; \on\, \bullet_{{1}}\, \circ_{{1}}\, \off\;\on  \ldots \bullet_{{j}}\, \circ_{j})} =  \bullet_{{i_{1}}}\, \circ_{{j_{1}}}\, \off\; \on\, \bullet_{{i_{2}}}\, \circ_{{j_{2}}}\, \off\;\on \ldots \bullet_{i_{m}}\, \circ_{j_{m}}$$
$$\oview{(\bullet_{{0}}\, \circ_{{0}}\, \off\;\on \ldots \bullet_{{j}}\, \circ_{j}\, \off\,\on\, \bullet_{j+1})} = \bullet_{i_{1}}\, \circ_{j_{1}}\, \off\;\on \ldots \bullet_{{i_{m}}}\, \circ_{j_{m}}\, \off\,\on\, \bullet_{j+1}$$
$$\pview{(\bullet_{{0}}\, \circ_{{0}}\, \off\;\on \ldots \bullet_{{j}}\, \circ_{j}\, \off\,\on\, \bullet_{j+1})} = \pview{\bullet_{0}\,\circ_{0}\,\bullet_{1}\,\circ_{1}\ldots \bullet_{j}\,\circ_{j}\,\bullet_{j+1}}$$
$$\pview{(\bullet_{{0}}\, \circ_{{0}}\, \off\; \on\, \bullet_{{1}}\, \circ_{{1}}\, \off\;\on  \ldots \bullet_{{j}}\, \circ_{j} \off)} =  \on^{(i_{1})}\, \off^{(j_1)}\, \ldots\ \on^{(i_{m})}\, \off^{(j_m)}$$

\item $s^{\medtriangledown}$ is a visible play over $\mathbb{A}^{+}$. 
\end{enumerate}
\end{proposition}
\begin{proof}
\mbox{}
\begin{enumerate}
\item
 By induction on $j$.  Let
 $$r=\bullet_{{0}}\, \circ_{{0}}\, \off\; \on\, \bullet_{{1}}\, \circ_{{1}}\, \off\;\on  \ldots \bullet_{{j}}\, \circ_{j}$$
and let $\efrom{\circ_{j}}=(\bullet_{i}, \circ_{j})$. Then
 $$r=\bullet_{{0}}\, \circ_{{0}}\, \off\;\on\, \bullet_{{1}}\, \circ_{{1}}\, \off\;\on\ldots \bullet_{{i-1}}\, \circ_{i-1}\, \off\, \on\, \bullet_{i}\ldots \circ_{j}$$
By definition of $O$-view
$$\oview{r}=\oview{(\bullet_{{0}}\, \circ_{{0}}\, \off\;\on\, \bullet_{{1}}\, \circ_{{1}}\, \off\;\on \ldots \bullet_{{i-1}}\, \circ_{i-1})}\, \off\,\on\, \bullet_{i}\, \circ_{j}$$
Letting 
$$\oview{\bullet_{0}\,\circ_{0}\,\bullet_{1}\,\circ_{1}\ldots \bullet_{i-1}\, \circ_{i-1}}=\bullet_{i_{1}}\, \circ_{j_{1}}\ldots\bullet_{{j_{n}}}\, \circ_{j_{n}}$$
we have by definition of $O$-view and induction hypothesis
$$\oview{\bullet_{0}\,\circ_{0}\,\bullet_{1}\,\circ_{1}\ldots \bullet_{j}\,\circ_{j}}=\bullet_{i_{1}}\, \circ_{j_{1}}\ldots\bullet_{{j_{n}}}\, \circ_{j_{n}}\,\bullet_{i}\circ_{j}$$
$$\oview{r}=\bullet_{{i_{1}}}\, \circ_{{j_{1}}}\, \off\; \on  \ldots \bullet_{i_{n}}\, \circ_{j_{n}}\, \off\, \on \, \bullet_{i}\, \circ_{j}$$
which establishes the first equation of the thesis. The second one is established by
$$\begin{aligned} 
&\oview{(\bullet_{{0}}\, \circ_{{0}}\, \off\;\on \ldots \bullet_{{j}}\, \circ_{j}\, \off\,\on\, \bullet_{j+1})}\\
=& \oview{(\bullet_{{0}}\, \circ_{{0}}\, \off\;\on \ldots \bullet_{{j}}\, \circ_{j})}\, \off\,\on\, \bullet_{j+1}\\
= &\,\bullet_{i_{1}}\, \circ_{j_{1}}\, \off\;\on \ldots \bullet_{{i_{m}}}\, \circ_{j_{m}}\, \off\,\on\, \bullet_{j+1}
\end{aligned}
$$
We now prove the third equation. Assume  $\efrom{\bullet_{j+1}}=(\circ_{i}, \bullet_{j+1})$. Then by induction hypothesis
$$\begin{aligned} 
&\pview{(\bullet_{{0}}\, \circ_{{0}}\, \off\;\on \ldots \bullet_{{j}}\, \circ_{j}\, \off\,\on\, \bullet_{j+1})}\\
=& \pview{(\bullet_{{0}}\, \circ_{{0}}\, \off\;\on \ldots \on\;\off\,\bullet_{{i}})}\, \circ_{i}\, \bullet_{j+1}\\
=&  \pview{(\bullet_{{0}}\, \circ_{{0}} \ldots \bullet_{{i}})}\, \circ_{i}\, \bullet_{j+1}\\
=&  \pview{(\bullet_{{0}}\, \circ_{{0}} \ldots \bullet_{{j}}\, \circ_{j}\, \bullet_{j+1})}
\end{aligned}
$$
Finally, let us prove the fourth equation. Let
 $$r=\bullet_{{0}}\, \circ_{{0}}\, \off\; \on\, \bullet_{{1}}\, \circ_{{1}}\, \off\;\on  \ldots \bullet_{{j}}\, \circ_{j} \off$$
and let $\efrom{\circ_{j}}=(\bullet_{i}, \circ_{j})$. Then
 $$r=\bullet_{{0}}\, \circ_{{0}}\, \off\;\on \ldots \bullet_{{i-1}}\, \circ_{i-1}\, \off\, \on^{(i)}\, \bullet_{i}\ldots \circ_{j}\,\off^{(j)}$$
 and there is an edge $(\on^{(i)}, \off^{(j)})$ in $r$.
By definition of $P$-view
$$\pview{r}=\pview{(\bullet_{{0}}\, \circ_{{0}}\, \off\;\on \ldots \bullet_{{i-1}}\, \circ_{i-1}\, \off )}\, \on^{(i)}\,\off^{(j)}$$
Letting 
$$\oview{\bullet_{0}\,\circ_{0}\,\bullet_{1}\,\circ_{1}\ldots \bullet_{i-1}\, \circ_{i-1}}=\bullet_{i_{1}}\, \circ_{j_{1}}\ldots\bullet_{{j_{n}}}\, \circ_{j_{n}}$$
we have by definition of $O$-view and induction hypothesis
$$\oview{\bullet_{0}\,\circ_{0}\,\bullet_{1}\,\circ_{1}\ldots \bullet_{j}\,\circ_{j}}=\bullet_{i_{1}}\, \circ_{j_{1}}\ldots\bullet_{{j_{n}}}\, \circ_{j_{n}}\,\bullet_{i}\circ_{j}$$
$$\pview{r}=\on^{(i_{1})}\, \off^{(j_1)} \ldots\ \on^{(i_{n})}\, \off^{(j_n)}\, \on^{(i)}\,\off^{(j)}$$
which is the thesis.

\item By Definition \ref{def-gadget}
$$s^{\medtriangledown}=  \bullet_{{0}}\, \circ_{{0}}\, \off\; \on\, \bullet_{{1}}\, \circ_{{1}}\, \off\;\on  \ldots \bullet_{{k}}\, \circ_{k}\, \off\;\on\, \bullet_{k+1}$$
Let $r\, a$ be any initial segment of $s^{\medtriangledown}$, with $r$ non empty. 
 We must show that $a$ points to a move in $\pview{r}$ if $\lambda(a)=P$, and that $a$ points to a move in $\oview{r}$ if $\lambda(a)=O$. The thesis is trivial when $a$ points to the immediately previous move, thus when $a=\on$ or $a=\off^{(0)}$. We have other three cases, according to the remaning possible shapes of $a$.
\begin{itemize}


\item If $a=\off^{(j)}$, with $j>0$, then $\lambda(a)=O$ and by Definition \ref{def-gadget} 
$$r=\ldots \on^{(i)}\, \bullet_{i}\ldots \circ_{j}$$
and the edges $(\bullet_{i}, \circ_{j})$ and $(\on^{(i)}, \off^{(j)})$ are in $s^{\medtriangledown}$.
Since  $\oview{r}=\ldots \on^{(i)}\,\bullet_{i}\,\circ_{j}$, we are done.

\item If $a=\circ_{0}$, then $\lambda(a)=P$, thus $\pview{r}=\bullet_{0}$ and we are done. If $a=\circ_{j+1}$, then $\lambda(a)=P$ and by the third equation of point 1., we have 
$$\pview{r}=\pview{(\bullet_{{0}}\, \circ_{{0}}\, \off\;\on \ldots \bullet_{{j}}\, \circ_{j}\, \off\,\on\, \bullet_{j+1})} = \pview{\bullet_{0}\,\circ_{0}\,\bullet_{1}\,\circ_{1}\ldots \bullet_{j}\,\circ_{j}\,\bullet_{j+1}}$$
Since $s$ is visible, $\circ_{j+1}$ points in $s$ -- and thus in $s^{\medtriangledown}$ -- to some $\bullet_{i}$ in $\pview{r}$, with $0\leq i\leq j+1$, and we are done.
\item If $a=\bullet_{j+1}$, then $\lambda(a)=O$ and as a consequence of the first equation of  point 1., 
we have 
$$\oview{r}=\oview{(\bullet_{{0}}\, \circ_{{0}}\, \off\;\on \ldots \bullet_{{j}}\, \circ_{j}\, \off\,\on)} = \,\bullet_{i_{1}}\, \circ_{j_{1}}\, \off\;\on \ldots \bullet_{{j_{m}}}\, \circ_{j_{m}}\, \off\,\on$$
with 
$$\oview{\bullet_{0}\,\circ_{0}\ldots \bullet_{j}\,\circ_{j}}=\bullet_{i_{1}}\, \circ_{j_{1}}\ldots\bullet_{{j_{m}}}\, \circ_{j_{m}}$$
Since $s$ is visible, $\bullet_{j+1}$ points in $s$ -- and thus in $s^{\medtriangledown}$ --  to some $\circ_{j_{n}}$, with $1\leq n\leq m$, and we are done.
\end{itemize}
\end{enumerate}
\end{proof}
Thanks to the results of Section \S\ref{sec-maxmin}, we know that in order to compute the backtracking level of a visible play, it is sufficient to compute the backtracking level of its views. Thus we start by proving that for every $O$-view $s$, the transformation mapping $s$ into $s^{\medtriangledown}$ effectively lowers the backtracking level of $O$ to $d-2$, assuming the arena is of depth $d\geq 2$. 
\begin{lemma}\label{lem-backview}
Let $s=\bullet_{0}\,\circ_{0}\,\bullet_{1}\,\circ_{1}\ldots \bullet_{k}\,\circ_{k}\, \bullet_{k+1}$ be a visible play over the arena $\mathbb{A}$ and $d\geq 2$ the depth of $\mathbb{A}$. Assume that for all $i$, the edge $(\bullet_{i},\circ_{i})$ is in $s$. 
  Then the backtracking level of $O$ in $s^{\medtriangledown}$ is less than or equal to $d-2$. 
\end{lemma}
\begin{proof}  
By Definition \ref{def-gadget}
$$s^{\medtriangledown}=  \bullet_{{0}}\, \circ_{{0}}\, \off\; \on\, \bullet_{{1}}\, \circ_{{1}}\, \off\;\on  \ldots \bullet_{{k}}\, \circ_{k}\, \off\;\on\, \bullet_{k+1}$$
We first prove that if $0<i<k+1$, then $\efrom{\bullet_{j}}$ is inactive for every $j\geq i$ in the following initial segment of $s^{\medtriangledown}$:
$$r\, =\, \bullet_{{0}}\, \circ_{{0}}\, \off\ldots \on\, \bullet_{{i}}\, \circ_{{i}}\, \off  \ldots \on\,\bullet_{{j}}\, \circ_{j}\, \off\; \on$$
We argue by induction on $j-i$.  By Definition \ref{def-gadget}, for $i\leq l\leq j$, the sequence $\on\, \bullet_{{l}}\, \circ_{{l}}\, \off$, with justification pointers, must have the following shape:
 $$
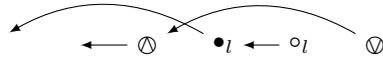
\begin{tikzpicture}
\node (a) at (1,0) {$$};
\node (b) at (2,0) {$\on$};
\node (c) at (3,0) {$\bullet_{l}$};
\node (d) at (4,0) {$\circ_{l}$};
\node (e) at (5,0) {$\off$};
\draw[-latex, right]  (b) edge (a);
\draw[-latex,bend right]  (c) edge (x);
\draw[-latex, right]  (d) edge (c);
\draw[-latex,bend right]  (e) edge (b);
\end{tikzpicture}$$
In particular the sequence $\on\, \bullet_{{j}}\, \circ_{{j}}\, \off$ has the following shape:
 $$
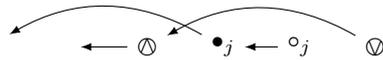
\begin{tikzpicture}
\node (a) at (1,0) {$$};
\node (b) at (2,0) {$\on$};
\node (c) at (3,0) {$\bullet_{j}$};
\node (d) at (4,0) {$\circ_{j}$};
\node (e) at (5,0) {$\off$};
\draw[-latex, right]  (b) edge (a);
\draw[-latex,bend right]  (c) edge (x);
\draw[-latex, right]  (d) edge (c);
\draw[-latex,bend right]  (e) edge (b);
\end{tikzpicture}$$
Since the above $j$-th occurrence of $\efrom{\off}$ is active in $r$, we infer that $\efrom{\bullet_{j}}$ is inactive in $r$. Moreover, for every $l$ such that $i< l < j$, by induction hypothesis $\efrom{\bullet_{l}}$ is inactive in $r$.   Now consider the sequence $\on\, \bullet_{{i}}\, \circ_{{i}}\, \off$:
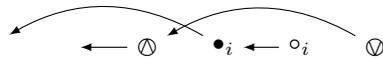
$$\begin{tikzpicture}
\node (a) at (1,0) {$$};
\node (b) at (2,0) {$\on$};
\node (c) at (3,0) {$\bullet_{i}$};
\node (d) at (4,0) {$\circ_{i}$};
\node (e) at (5,0) {$\off$};
\draw[-latex, right]  (b) edge (a);
\draw[-latex,bend right]  (c) edge (x);
\draw[-latex, right]  (d) edge (c);
\draw[-latex,bend right]  (e) edge (b);
\end{tikzpicture}$$
Since every edge $e$ of $r$ such that $\efrom{\off}\cross e$ must be of the form  $e=\efrom{\bullet_{l}}$, with $i<l\leq j$ and $e$ is inactive in $r$, we conclude that $\efrom{\off}$ is active in $r$, and finally that $\efrom{\bullet_{i}}$ is inactive in $r$.\\ 
From what we have just proved we conclude that for no $i,j$ it is the case that  $\efrom{\bullet_{i}}\inact \efrom{\bullet_{j}}$ in $s^{\medtriangledown}$. Hence, every maximal inactivation chain $c$ in $s^{\medtriangledown}$ that determines the backtracking level of $O$ is contained in a chain having one of the two following shapes:
$$\efrom{\bullet_{i_{1}}}\inact \efrom{\off^{(i_{1})}}\inact\efrom{\bullet_{i_{2}}}\inact \efrom{\off^{(i_{2})}}\inact\ldots\inact \efrom{\bullet_{i_{m}}}\inact \efrom{\off^{(i_{m})}}$$
$$ \efrom{\off^{(i_{1})}}\inact\efrom{\bullet_{i_{2}}}\inact \efrom{\off^{(i_{2})}}\inact\ldots\inact \efrom{\bullet_{i_{m}}}\inact \efrom{\off^{(i_{m})}}$$
In the first case, there is a play
$$\bullet_{i_{1}}\leftarrow \circ_{i_{1}}\leftarrow \bullet_{i_{2}}\leftarrow \circ_{i_{2}}\leftarrow \ldots\leftarrow \bullet_{i_{m}}\leftarrow \circ_{i_{m}}$$
Since $\bullet_{i_{1}}$ points backward, it cannot be the first move and thus points to a Player move, which in turn must point to an Opponent move; therefore, $\bullet_{i_{1}}$  must be the last move of a simple play over $\mathbb{A}$ of at least $3$ moves. Since the arena depth is  $d$, we conclude $2m\leq d-2$, which is the thesis. In the second case, there is a play 
$$ \circ_{i_{1}}\leftarrow \bullet_{i_{2}}\leftarrow \circ_{i_{2}}\leftarrow \ldots\leftarrow \bullet_{i_{m}}\leftarrow \circ_{i_{m}}$$
Now, $i_{1}=0$ is impossible, for $\off^{(0)}$ points to the immediately preceding move $\circ_{0}$ and it cannot be inactivated by any edge. Therefore, $\circ_{i_{1}}$ points to the immediately preceding Opponent move, which points to a Player move pointing to yet another Opponent move; this means that $\circ_{i_{1}}$ must be the last move of a simple play over $\mathbb{A}$ of at least $4$ moves. We conclude $2m-1\leq d-2$, which is the thesis. 
\end{proof}
We are now in a position to prove that for any visible play $s$ over $\mathbb{A}$, the backtracking level of $O$ in $s^{\medtriangledown}$ is less than or equal to the depth $d$ of the arena minus $2$.
\begin{theorem}[Backtracking Level Normalization]\label{thm-downback}
Let $s=\bullet_{0}\,\circ_{0}\,\bullet_{1}\,\circ_{1}\ldots \bullet_{k}\,\circ_{k}\, \bullet_{k+1}$ be a visible play over the arena $\mathbb{A}$ and $d\geq 2$ the depth of $\mathbb{A}$. 
  Then the backtracking level of $O$ in $s^{\medtriangledown}$ is less than or equal to $d-2$. 
\end{theorem}
\begin{proof}
By Definition \ref{def-gadget}
$$s^{\medtriangledown}=  \bullet_{{0}}\, \circ_{{0}}\, \off\; \on\, \bullet_{{1}}\, \circ_{{1}}\, \off\;\on  \ldots \bullet_{{k}}\, \circ_{k}\, \off\;\on\, \bullet_{k+1}$$
By Proposition \ref{prop-gadgetvis}, the $O$-views of the initial segments of $s^{\triangledown}$ that we have to consider are of the form
$$\begin{aligned} 
&\oview{(\bullet_{{0}}\, \circ_{{0}}\, \off\;\on \ldots \bullet_{{j}}\, \circ_{j}\, \off\,\on\, \bullet_{j+1})}\\
= &\,\bullet_{i_{1}}\, \circ_{j_{1}}\, \off\;\on \ldots \bullet_{{i_{m}}}\, \circ_{j_{m}}\, \off\,\on\, \bullet_{j+1}
\end{aligned}$$
and 
$$\begin{aligned} 
&\oview{(\bullet_{{0}}\, \circ_{{0}}\, \off\;\on  \ldots \bullet_{{j}}\, \circ_{j}\, \off)}\\
= &\,\bullet_{{i_{1}}}\, \circ_{{j_{1}}}\, \off\;\on \ldots \bullet_{i_{m}}\, \circ_{j_{m}}\,\off
\end{aligned}
$$
with 
$$\oview{\bullet_{0}\,\circ_{0}\,\bullet_{1}\,\circ_{1}\ldots \bullet_{j}\,\circ_{j}}=\bullet_{i_{1}}\, \circ_{j_{1}}\ldots\bullet_{{j_{m}}}\, \circ_{j_{m}}$$
and thus, for $1\leq n\leq m$, the edges $(\bullet_{j_{n}}, \circ_{j_{n}})$ are in $s$ and so in $s^{\medtriangledown}$.
By Proposition \ref{prop-back2}, every inactivation chain in $s^{\medtriangledown}$ ending with an $O$-edge is already contained in the $O$-view of the initial segment of $s^{\medtriangledown}$ ending with that edge. 
Moreover, by Lemma \ref{lem-backview},  every $O$-view of any initial segment of $s^{\medtriangledown}$ terminating with $\bullet_{j+1}$ or $\off$ is of backtracking level less than or equal to $d-2$.   We conclude that the backtracking level of $O$ in $s^{\medtriangledown}$ is less than or equal to $d-2$.
\end{proof}
We conclude by proving a strengthened version of Theorem \ref{thm-B}, which guarantees that the tower of exponentials defining the complexity of strategy interactions is never higher than the arena depth minus $2$.   

\begin{theorem}[Refined-Min-Backtracking-Theorem]\label{thm-B2}
Let $\sigma$ be a Player $n$-backtracking strategy bounded by $k$, $\tau$  an Opponent $m$-backtracking strategy bounded by $k$, $d\geq 2$ the depth of the arena $\mathbb{A}$ and $b=\mathsf{min}(n,m, d-2)$. Assume $s\in\sigma\star\tau$ and $s=s_{0}\, s_{1}\ldots s_{l}$. Then 
$$l\ \leq \ \underbrace{(2k)^{(2k)^{.^{.^{(2k)^{2k}}}}}}_{b+1}\leq  \underbrace{2^{2^{.^{.^{2^{2k(\log 2k)\cdot 2}}}}}}_{b+1}=2_{b}^{2k (\log 2k)\cdot 2} $$
\end{theorem}
\begin{proof}
By Theorem \ref{thm-maxmin}, $s$ is visible and of real backtracking level $b$. If $n< d-2$ or $m<d-2$, the thesis follows by Theorem \ref{thm-B}. Otherwise, by Theorem \ref{thm-downback}, $O$ is of backtracking level no greater than $d-2$ in $s^{\medtriangledown}$.  By Definition \ref{def-stratvisint} of $\sigma\star\tau$, every initial segment $s'$ of $s$ is in $\sigma$ or $\tau$ and since both strategies are bounded by $k$, $\pview{s'}\leq k$ and $\oview{s'}\leq k$. From the characterization of views of Proposition \ref{prop-gadgetvis}, it easily follows that for every initial segment $s'$ of $s^{\triangledown}$, it holds $\pview{s'}\leq 2k$ and $\oview{s'}\leq 2k$.  By Theorem \ref{thm-abstB} applied to $s^{\medtriangledown}$, and since $s$ is shorter than $s^{\medtriangledown}$, we obtain the thesis.
\end{proof}

\textbf{Concluding remarks.} Theorem \ref{thm-B2} can also be derived in another way. When the minimum among the backtracking levels of the strategies is smaller than the depth of the arena minus $2$, one uses Theorem  \ref{thm-abstB}. When it is not the case, one can use Coquand's Trimming technique; with little adjustments, our main calculations would work just fine and one would obtain the correct upper bounds. The techniques of this section, however, are sharper: they open the possibility of optimizing strategies. For example, if just a proper initial segment of a strategy has excessive backtracking level, using our techniques one may very well take the overall backtracking level even below the arena depth, thus obtaining better bounds.

Our bounds, moreover, are sharp. This easily follows from Clairambault's lower bounds \cite{Clairambaultjour} or any others. Indeed, whenever a strategy debate takes as many steps as computed by a tower of exponentials having height $d-2$, we know that the minimum among the backtracking levels of the players  must be at least $d-2$, so it can be taken to be  exactly $d-2$.\\

\textbf{Acknowledgments}. This work was inspired by interesting conversations with Stefan Hetzl and Daniel Weller on their expansion trees with cut \cite{HetzlWeller}, which provided my motivation for studying in depth game semantical debates between strategies. I would also like to thank the first for his academic support during my first year at TU Wien.

\end{document}